\documentclass[11pt]{amsart}
\usepackage{amsmath}
\usepackage{graphicx}

\newtheorem{theorem}{Theorem}[section]
\newtheorem{lemma}[theorem]{Lemma}
\newtheorem{corollary}[theorem]{Corollary}
\newtheorem{conjecture}[theorem]{Conjecture}

\theoremstyle{definition}
\newtheorem{definition}[theorem]{Definition}

\theoremstyle{remark}
\newtheorem{remark}[theorem]{Remark}
\newtheorem{example}[theorem]{Example}

\def\R{\mathbb{R}}
\def\Z{\mathbb{Z}}

\newcommand{\A}{\mathcal{A}}
\newcommand{\tb}{\mathit{tb}}
\def\d{\partial}
\def\End{\operatorname{End}}

\def\p{d}

\let\oldmarginpar\marginpar
\renewcommand\marginpar[1]{\oldmarginpar{\raggedright\footnotesize #1}}

\begin{document}

\title[Satellites and Chekanov--Eliashberg algebras]{Satellites of Legendrian knots and representations of the Chekanov--Eliashberg algebra}
\author{Lenhard Ng}
\address{Mathematics Department, Duke University, Durham, NC 27708}
\email{ng@math.duke.edu}
\author{Dan Rutherford}
\address{Department of Mathematics, University of Arkansas, Fayetteville, AR 72701}
\email{drruther@uark.edu}

\begin{abstract}
We study satellites of Legendrian knots in $\R^3$ and their relation to the Chekanov--Eliashberg differential graded algebra of the knot.
In particular, we generalize the well-known correspondence
between rulings of a Legendrian knot in $\R^3$ and augmentations of its DGA by showing that the DGA has finite-dimensional representations if and only if there exist certain rulings of satellites of the knot.
We derive several
consequences of this result, notably that the question of existence of ungraded finite-dimensional representations for the DGA of a Legendrian knot depends only on the topological type and Thurston--Bennequin number of the knot.
\end{abstract}

\maketitle

\section{Introduction}

The satellite construction in knot theory produces new knot types from a given knot $K$ by considering the image of a solid torus knot (or link) $L \subset S^1 \times D^2$ inside a tubular neighborhood of $K$.   This construction provides a simple template for producing whole classes of knot invariants of $K$ since we can simply apply existing invariants to the various satellites of $K$.  As a well-known example, when this scheme is applied to the Jones polynomial, the associated class of satellite invariants are typically organized into a sequence of so-called colored Jones polynomials.   This wider collection of invariants generalizes the Jones polynomial to a family of quantum $\mathfrak{sl}_2$ invariants obtained by labeling $K$ by an arbitrary irreducible representation.

An analog of the satellite construction exists for Legendrian knots in $\R^3$ with its standard contact structure.  In this setting, a Legendrian satellite $S(K,L) \subset \R^3$ arises from a Legendrian pattern,  $L \subset J^1(S^1)$, and a Legendrian companion, $K \subset \R^3$, where $J^1(S^1)= T^*S^1\times\R$ denotes the $1$-jet space of $S^1$.  In this article, we study the effect of this Legendrian satellite operation on a pair of related invariants of Legendrian knots, the ruling polynomials, and the Chekanov--Eliashberg differential graded algebra (DGA).  All of the relevant definitions are recalled in Section 2.

A well-known result \cite{F, FI, Sabloff} in Legendrian knot theory asserts that for a Legendrian knot $K \subset \R^3$, the existence of a normal ruling of the front projection of $K$ is equivalent to the existence of an augmentation of the Chekanov--Eliashberg DGA of $K$.  (See Theorem \ref{thm:FIS} below.)  In addition, a relationship between ruling polynomials and the Kauffman and HOMFLY-PT knot polynomials (\cite{R}, see Theorem \ref{thm:Kauffman} below) shows that for either of these conditions to hold, the Thurston--Bennequin number $tb(K)$ of $K$ must be maximal.  Moreover, the existence of $1$- or $2$-graded augmentations or normal rulings depends only on the Thurston--Bennequin number and topological knot type of $K$.  In this article, we present generalizations of these results to finite-dimensional representations of the Chekanov--Eliashberg algebra (where an augmentation is simply a $1$-dimensional representation) and certain normal rulings of Legendrian satellites of $K$.

To give a more precise overview of our main results, we introduce some notation.  Denote the Chekanov--Eliashberg differential graded algebra of $K$ over $\Z/2$ by $(\A(K,*), \partial)$.  In our notation, $*$ refers to a chosen base point on $K$ which corresponds to a distinguished algebra generator $t$ measuring homology classes in $H_1(K)$.  Given a divisor $\p$ of twice the rotation number of $K$, a $\p$-graded representation of $(\A(K,*), \partial)$ consists of a $\Z/\p$-graded vector space $V$ together with a homomorphism of differential graded algebras $f: (\A(K,*), \partial) \rightarrow (\End(V), 0)$.  The requirements here are that $f\circ \partial = 0$ and $f$ preserves grading mod $\p$.  See Definition \ref{def:rep} below.

In Theorem \ref{thm:MainResult} we provide necessary and sufficient conditions for the existence of $\p$-graded representations of $(\A(K,*), \partial)$ of any fixed graded dimension in terms of so-called normal rulings of certain satellites of $K$ (see Section~\ref{ssec:rulings} for the definition of normal ruling).  We can give a particularly simple statement in the special case of $1$-graded representations:

\begin{theorem}[cf.\ Theorem~\ref{thm:MainResult}] \label{thm:intro} Let $K\subset\R^3$ be a Legendrian knot. Then the DGA $\A(K, *)$ admits a $1$-graded representation of dimension $n$ if and only if the satellite of $K$ with an $n$-stranded Legendrian full twist, $\mathit{tw}_n$, has a normal ruling.
\end{theorem}

\noindent
Topologically, the satellite in Theorem~\ref{thm:intro} is the $n$-component link given by $n$ parallel copies of $K$ with respect to framing coefficient $tb(K)+1$.
We also generalize Theorem~\ref{thm:MainResult} to give an explicit relation between satellites with more general patterns than $\mathit{tw}_n$ and representations of the DGA of particular sorts; see Theorem~\ref{thm:tech}.

In order to prove Theorems~\ref{thm:intro} and~\ref{thm:tech}, we need to study normal rulings of general Legendrian satellites $S(K,L)$ in terms of the pattern $L \subset J^1(S^1)$ and companion $K \subset\R^3$, and this is an interesting subject in its own right.
Many normal rulings of a satellite $S(K,L)$ do not reflect any significant aspect of $K$.  Indeed, an explicit construction (Theorem \ref{thm:bijection} below) shows that any normal ruling of the pattern $L$ may be extended to a normal ruling of $S(K,L)$, so the satellite always has at least as many normal rulings as $L$.  (See Corollary \ref{cor:estimate}.)  If $S(K,L)$ happens to have  more normal rulings than $L$ we say that $K$ is {\it $L$-compatible}.  The question of existence of normal rulings of $K$ then naturally generalizes to the question of $L$-compatibility.  This turns out to be related to the existence of representations of $\A(K,*)$: $\A(K,*)$ has a finite dimensional representation if and only if we can find a pattern $L \subset J^1(S^1)$ such that $K$ is $L$-compatible, and for particular $L$ we can classify which sorts of representations of $\A(K,*)$ are equivalent to $L$-compatibility.
The approach here is to reduce from the case of an arbitrary pattern $L$ to the case when $L$ is a product (disjoint union) of particularly simple patterns $A_k$ called {\it basic fronts}. See Theorems~\ref{thm:tech} and~\ref{thm:arbitrary}.

We also extend other results regarding augmentations and rulings to the setting of finite-dimensional representations. For instance, we have the following generalization of the result that the existence of an augmentation of the DGA implies maximal Thurston--Bennequin number and depends only on topological knot type and $tb$:

\begin{theorem}[cf.\ Theorem~\ref{thm:topinvt}]
If the DGA $(\A(K,*),\d)$ has an (ungraded) finite-dimensional representation, then $K$ maximizes $tb$ within its topological knot type. Furthermore, the existence of such a representation depends only on $tb(K)$ and the topological type of $K$.
\end{theorem}

\noindent
This result again makes use of the connection with the Kauffman and HOMFLY-PT polynomials as in the work of the second author \cite{R}. There is also a relation to a conjecture by the first author \cite{NgCLI} about the topological invariance of the so-called abelianized characteristic algebra; see Section~\ref{ssec:repcond}.

We note that there are examples of Legendrian torus knots with maximal $tb$ that admit $2$-dimensional representations but not $1$-dimensional representations.  This was observed by Sivek in \cite{Siv}; the question of existence of finite dimensional representations was raised in the same article, and is discussed in this context in Section~\ref{sec:2d}.  At this time, it is an open question if there are knots that admit $3$-dimensional representations but have no representations of dimension $1$ or $2$.

\begin{remark}
Our work indicates that in a certain concrete way, information about the DGA of various satellites of $K$ is already encoded in the Chekanov--Eliashberg algebra of $K$ itself, though in an algebraically complicated manner. (By comparison, note e.g.\ that the colored Jones polynomials for a smooth knot are not determined by the Jones polynomial.) For example, Theorem \ref{thm:intro} shows that augmentations of the satellite of $K$ with a full twist correspond to $n$-dimensional representations of $\A(K,*)$.  It is interesting to ask how much this situation persists to the plethora of other invariants derived from $\A(K,*)$.  For instance, can the collection of linearized homology groups of satellites of $K$  be recovered from $\A(K,*)$?
\end{remark}

We conclude this section by outlining the rest of the paper.
In Section~\ref{sec:background}, we recall necessary background about the satellite construction, normal rulings, and the Chekanov--Eliashberg algebra.  Section~\ref{sec:rulings} focuses on normal rulings of satellite links.  A restricted class of {\it reduced normal rulings} is introduced as they have a particularly close connection with representations of $\A(K,*)$.

Section~\ref{sec:reps} contains theorems connecting finite-dimensional representations and normal rulings of satellites, including most of the results discussed in this introduction.  Our most precise result, Theorem~\ref{thm:tech}, gives, in the case of a pattern $L \subset J^1(S^1)$ without cusps, an equivalence between the existence of reduced rulings of $S(K,L)$ and finite dimensional representations of $\A(K,*)$ in which the distinguished generator $t$ has matrix related to the path matrix of $L$ introduced by K\'alm\'an \cite{Ka}.  

Finally, Section~\ref{sec:2d} provides a detailed treatment of a special case, the question of existence of $2$-dimensional representations (with particular restrictions on the image of $t$).  A sufficient condition for $A_2$-compatibility is given in Theorem~\ref{prop:A2}, and the case of knot types with $10$ or fewer crossings is addressed completely.

\subsection*{Acknowledgments}

We thank Brad Henry, Steven Sivek, and Michael Sullivan for interesting conversations related to this work. The first author was partially supported by NSF CAREER grant DMS-0846346. The second author thanks the Max-Planck-Institut f\"ur Mathematik in Bonn for providing an excellent research atmosphere during a portion of this work.

\section{Background}
\label{sec:background}

In this section, we give background on Legendrian links in $\R^3$ and $J^1(S^1)$, the Legendrian satellite construction, normal rulings, the Chekanov--Eliashberg differential graded algebra, and assorted other constructions that will be necessary for the remainder of the paper.

\subsection{Legendrian links}
We consider Legendrian links in $\R^3$ and in an open solid torus $S^1\times \R^2$ with the contact structure provided in either case by the kernel of $dz - y\, dx$.  From the point of view of contact geometry, these spaces are perhaps more naturally viewed as the $1$-jet spaces of the line and circle respectively.  Correspondingly, we will usually use $J^1(S^1)$ to denote $S^1\times \R^2$ with this contact structure.

A Legendrian link $L$ in a $1$-jet space $J^1(M)$ can be recovered
from its image in $M\times\R$ under the projection $T^*M\times\R \to
M\times\R$; this image is referred to as the {\it front projection} or
{\it front diagram} of $L$.  We will use the same notation for a link
and its front diagram, but will point out the distinction when necessary.  In the case when $L \subset J^1(S^1)$,
the $x$-coordinate is circle-valued, so the front projection is a
subset of an annulus.  We write this annulus $S^1\times\R$ as
$[0,1]\times\R$ with the lines $\{0\}\times\R$ and $\{1\}\times\R$
identified, and we will often view the front projection of a
Legendrian link in $J^1(S^1)$ as a subset of $[0,1]\times\R$.

Generically, front projections are unions of closed curves in the $xz$-plane or annulus which are immersed away from semi-cubical cusp singularities and one-to-one except for transverse double points.  In addition, vertical tangencies cannot occur.  Conversely, any collection of curves of this type lifts to a Legendrian link.  Two Legendrian links, $L_0$ and $L_1$, are {\it Legendrian isotopic} if there is a smooth isotopy, $L_t$, connecting them with $L_t$ Legendrian for all $t\in [0,1]$.  The Legendrian isotopy $L_t$ may always be chosen so that the front projections of the $L_t$ are generic except for a finite number of {\it Legendrian Reidemeister moves}; see Figure~\ref{fig:LReid}.  A Legendrian isotopy of this type will be referred to as a {\it generic isotopy}.

\begin{figure}
\centerline{\includegraphics[width=6.5in]{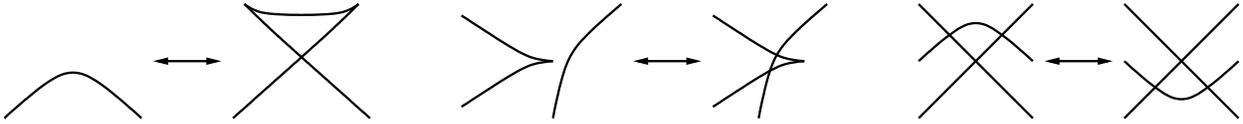}}
\caption{The Legendrian Reidemeister moves, from left to right: I, II, III.}
\label{fig:LReid}
\end{figure}

A Legendrian link $L$ has a framing arising from the contact structure.  There is a Legendrian isotopy invariant $\tb(L) \in \Z$ which equals the corresponding framing number in the case that $L$ is null-homologous.  In general, we define $\tb(L)$ via a generic front projection for $L$ by $\tb(L) = w(L) - \frac{1}{2} c(L)$ where $w(L)$ is the writhe of the projection and $c(L)$ is the number of cusps.

For an oriented (connected) Legendrian knot $K$, a second integer-valued invariant, the rotation number $r(K)$, is provided by the winding number of the tangent to $K$ around $0$ in the contact planes.  This is computed from a front diagram as $\frac{1}{2}(d(K) - u(K))$ where $d(K)$ (resp. $u(K)$) denotes the number of downward (resp. upward) oriented cusps.  For a multi-component link $L$, we will adopt the convention of taking $r(L)$ to be the greatest common divisor of the rotation numbers of the components of $L$.

\medskip

\subsection{Maslov potentials} \label{ssec:potentials}
The following additional structure on a front diagram may be viewed as a generalization of an orientation.

\begin{definition}  Let $\p$ be a divisor of $2 r(K)$.  A {\it
    $\p$-graded Maslov potential}, $\mu$, for $L$ is a function from
  the front diagram of $L$ to $\Z/\p$ which is constant except at cusp
  points where it increases by $1$ when moving from the lower branch
  of the cusp to the upper branch.  See Figure \ref{fig:MasP}. If $\p$ is even we
  assume in addition that $\mu$ is even along strands where the
  orientation of $L$ is in the positive $x$-direction.

\begin{figure}
\centerline{\includegraphics[scale=1]{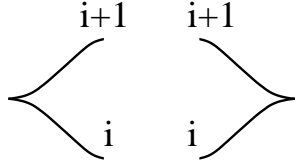}}
\caption{Requirements on a Maslov potential near cusps.}
\label{fig:MasP}
\end{figure}

A \textit{$\p$-graded Legendrian link} is a pair $(L, \mu)$ consisting of a Legendrian link $L$ with chosen $\p$-graded Maslov potential $\mu$.  Maslov potentials may be continued in a canonical way during any of the Legendrian Reidemeister moves.   Two $\p$-graded Legendrian links $(L_i, \mu_i)$, $i = 0,1$, are {\it Legendrian isotopic} if there is a generic isotopy from $L_0$ to $L_1$ which takes $\mu_0$ to $\mu_1$.
\end{definition}

We note that for a single-component Legendrian knot $K$, Maslov potentials are unique up to the addition of an overall constant. In fact, more is true: if $\mu$ and $\mu'$ are $\p$-graded Maslov potentials on $K$, then $(K,\mu)$ and $(K,\mu')$ are Legendrian isotopic, assuming if $\p$ is even that $\mu$ and $\mu'$ determine the same orientation on $K$. See Remark~\ref{rmk:uniqueMaslov}.

\medskip

\subsection{Legendrian satellites} \label{sub:satellite}

The following construction first appears in the literature in \cite{NgTr} where we refer the reader for additional details.
Let $(K, *) \subset \R^3$ be an oriented (connected) Legendrian knot with chosen base point, $*$, and $L \subset J^1(S^1)$ a link.  Using this information, we form a new link $S(K,L) \subset \R^3$ whose Legendrian isotopy type depends only on the Legendrian isotopy types of $K$ and $L$.  The knot $K$ is referred to as the {\it companion}; $L$ is the {\it pattern}; and $S(K,L)$ is the resulting {\it satellite}.

Say that $(K,*)$ is in \textit{general
    position} if its front has generic singularities and $*$ lies away
  from these; say that $L \subset J^1(S^1)$ is in general position if its
  front, viewed as a subset of $[0,1] \times \R$ with ends identified,
  has generic singularities, all away from $x=0$. Then we can define
  $S(K,L)$ diagrammatically.

Let $n$ denote the number of intersection points of the front diagram of $L$ with the vertical line $x=0$.  We begin by forming a link whose front projection is obtained by taking $n$ copies of $K$ and shifting the $z$-coordinate of each successive copy downward by a small amount.  This link is referred to as the $n$-copy of $K$.  Next, we insert the front projection of $L$ into the $n$-copy of $K$ at the location of the base point.  To do this we view the front projection of $L$ as a subset of $[0,1]\times \R$ and scale appropriately so that the $n$ intersection points of $L$ with $x=0$ and $x=1$ line up with the $n$ parallel strands in the $n$-copy of $K$.  Furthermore, the scaling should be carried out so that $L$ does not intersect other parts of the $n$-copy of $K$.  If $K$ is oriented to the right at $*$ then we insert $L$ directly into the $n$-copy.  However, if $K$ is oriented to the left at $*$ we instead insert the reflection of $L$ across a vertical axis.  See Figure \ref{fig:SatEx}.

\begin{figure}
\[
\begin{array}{ccc} 
\includegraphics[scale=.3]{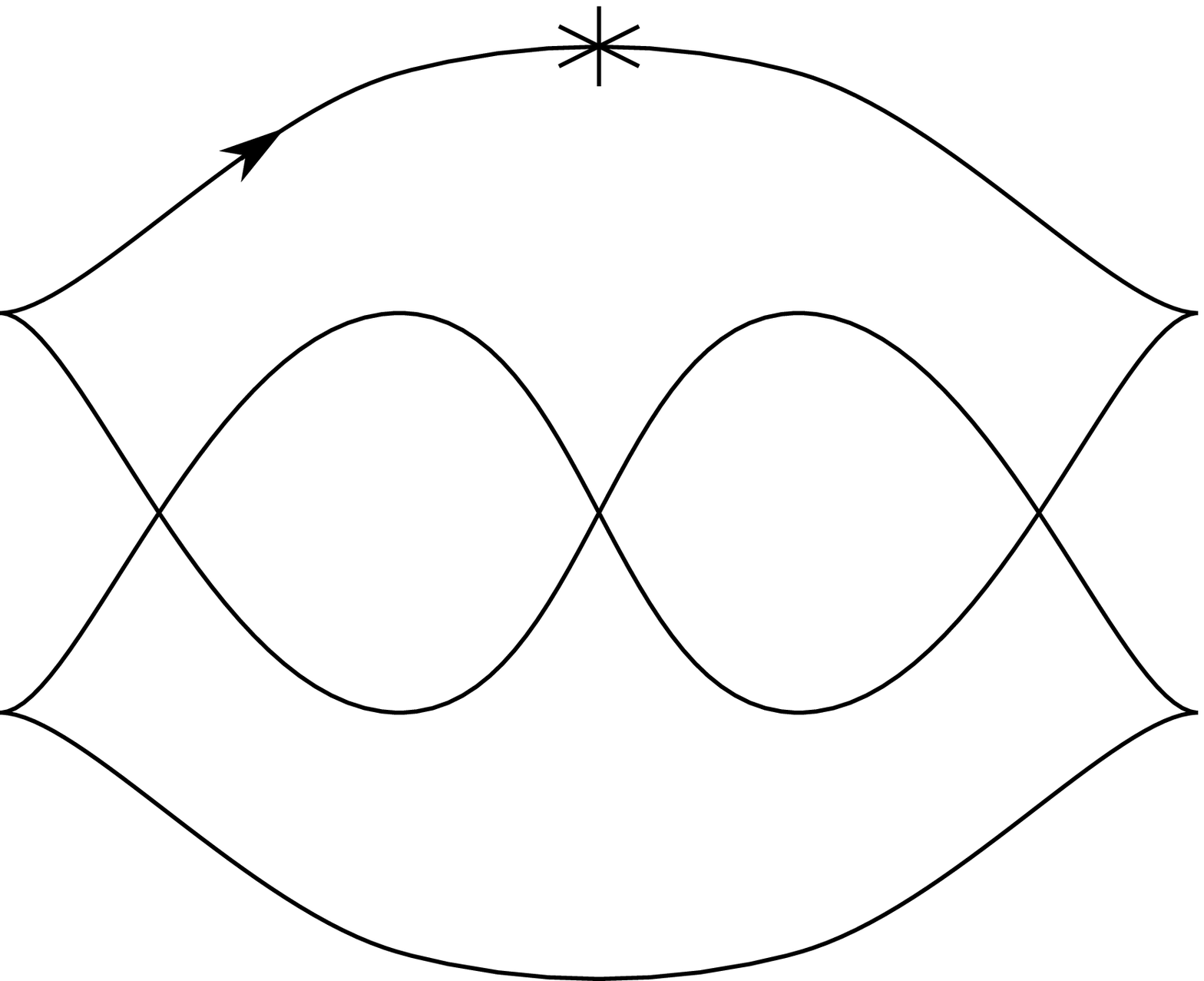} & \quad &\raisebox{2.1 ex}{\includegraphics[scale=.6]{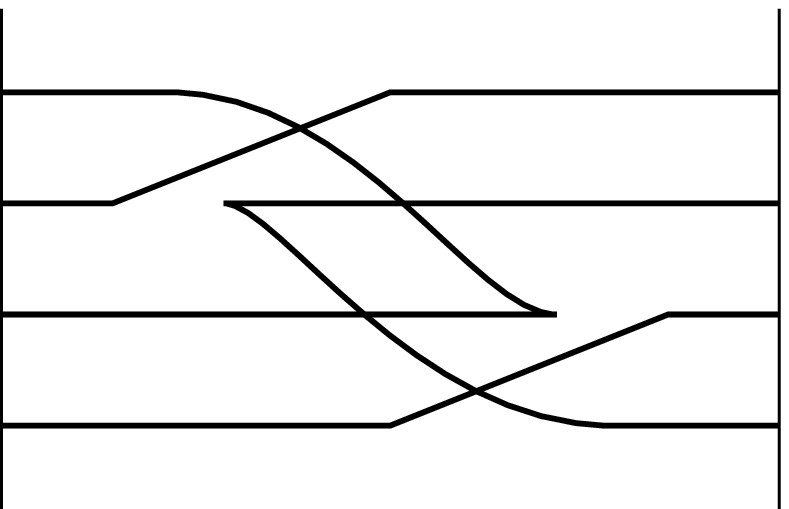}} \\
K & \quad & L
\end{array}
\]
\centerline{\includegraphics[scale=.4]{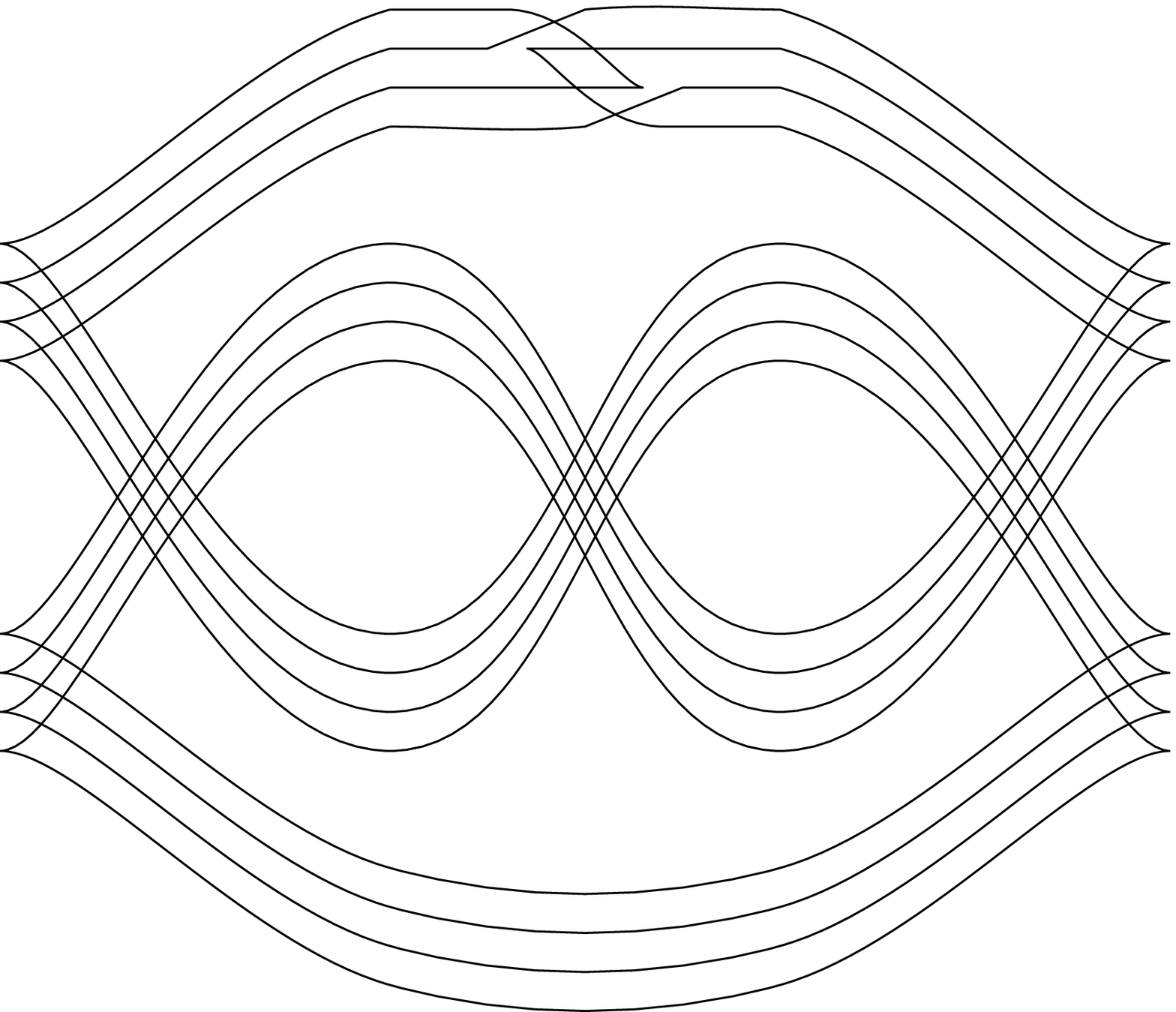}}
\centerline{$S(K,L)$}
\caption{The Legendrian satellite construction.}
\label{fig:SatEx}
\end{figure}

\begin{remark}  It will often be convenient to have an enumeration of
  the $n$ parallel copies of $K$ that make up the satellite.  In this
  article, we adopt the convention of labeling strands of $S(K,L)$
  corresponding to a single strand of $K$ from $1$ to $n$ from top to bottom,
with the index label increasing as the $z$-coordinate decreases.
\end{remark}

  If $\p$ is a common divisor of $2 r(K)$ and $2 r(L)$, then the choice of $\p$-graded Maslov potentials, $\mu$ and $\eta$, for $K$ and $L$ gives rise to a $\p$-graded Maslov potential, $\widetilde{\mu}$, for $S(K,L)$ as follows.  At the location of the base point $*$ where $L$ is inserted, $\widetilde{\mu}$ is the sum of $\mu(*)$ and $\eta$.  Since $K$ is connected, this uniquely characterizes $\widetilde{\mu}$.  Indeed, let $\eta_i$ denote the value of $\eta$ on the $i$-th strand of $L$ at $x=0$.  Then  at the $i$-th strand of $S(K,L)$ corresponding to the point $k \in K$, $\widetilde{\mu}= \mu(k) + \eta_i$.

In \cite{NgTr}, generic Legendrian isotopies are given to show that the Legendrian isotopy type of $S(K,L)$ depends only on the Legendrian isotopy types of $K$ and $L$ and, in particular, is independent of the choice of base point $*$.  Paying attention to Maslov potentials in the proof shows that the $\p$-graded Legendrian isotopy type of $\left( S(K,L), \widetilde{\mu}\right)$ depends only on that of $(K,\mu)$ and $(L, \eta)$.

\begin{remark}
\label{rem:sat}
 The analogous construction for smooth knots requires a choice of
 framing for $K$ in order to produce a satellite $S(K,L)$ whose
 isotopy type is well defined.  
If $K$ and $L$ are Legendrian, such a framing is given by the contact
framing for $K$, which has framing coefficient $tb(K)$ relative to the
Seifert framing; thus in this case the smooth knot type
of $S(K,L)$ depends only on $tb(K)$ along with the underlying smooth knot
types of $K$ and $L$.
\end{remark}

\begin{remark}
\label{rmk:uniqueMaslov}
As mentioned in Section~\ref{ssec:potentials}, one can use the satellite construction to give an easy proof of the following result.

\begin{theorem}
Let $K$ be a connected Legendrian knot in $\R^3$ with two $\p$-graded Maslov potentials $\mu$ and $\mu'$, where $\mu'-\mu = k$ for some constant $k$, with the additional stipulation that if $\p$ is even, then $k$ is also even. Then $(K,\mu)$ and $(K,\mu')$ are Legendrian isotopic.
\end{theorem}

\begin{figure}
\centerline{\includegraphics[scale=0.35]{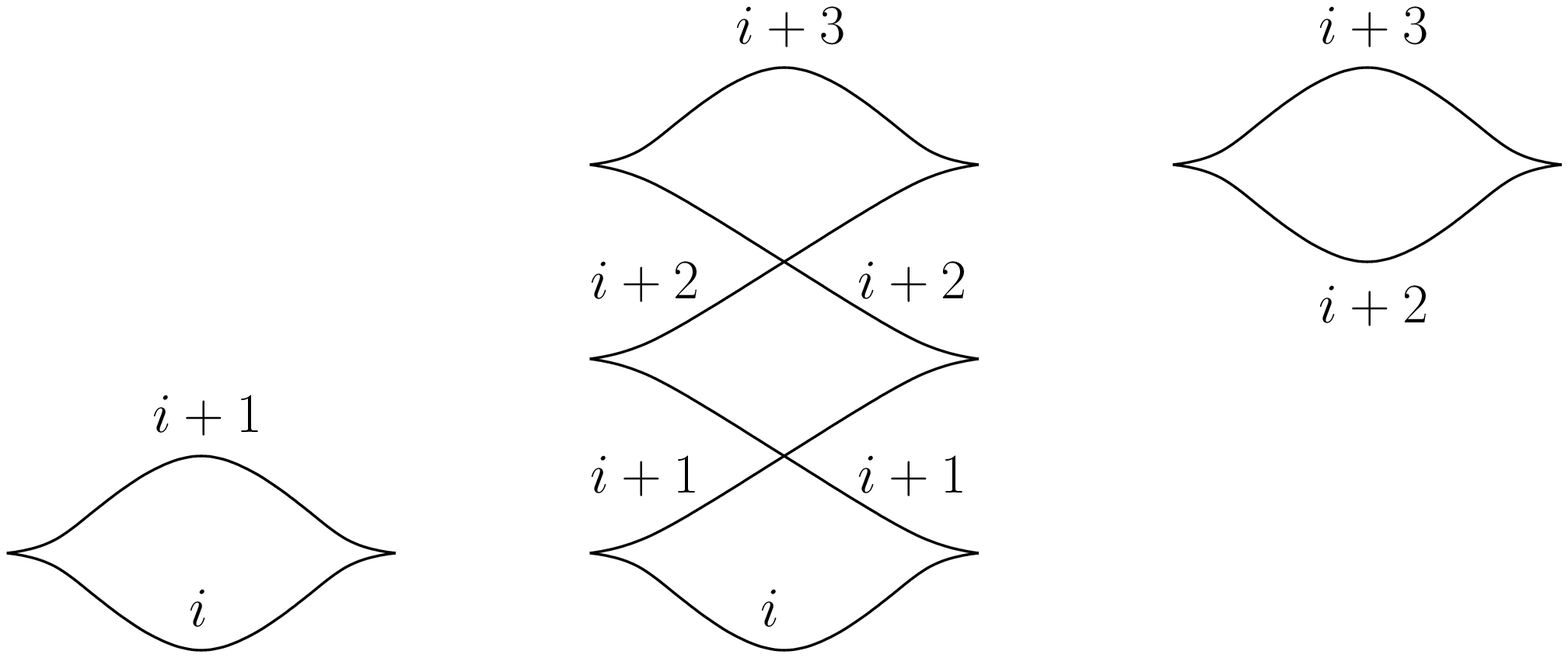}}
\caption{Changing Maslov potential for the unknot via Legendrian isotopy.}
\label{fig:UnknotMaslov}
\end{figure}

\begin{proof}
It suffices to prove the theorem when $k=2$, since a general $k$ is a multiple of $2$ (note that when $\p$ is odd, every integer is even mod $\p$). If in addition $K$ is the standard unknot $U$, then the theorem holds since we can apply two Legendrian Reidemeister I moves, then undo these moves, as shown in Figure~\ref{fig:UnknotMaslov}.

For a general knot $K$, note that $K$ can be expressed as a satellite of the unknot, $K = S(U,L)$ for some $L \subset J^1(S^1)$: simply perform a Reidemeister I move somewhere along the knot, and then the new loop is the unknot $U$. We can choose Maslov potentials on $U$ and $L$ so that the Maslov potential $\mu$ on $K$ is the sum of these two, in the sense discussed in this subsection. By the theorem for the unknot, there is a Legendrian isotopy on $U$ that changes the Maslov potential on $U$ by $2$; then the induced Legendrian isotopy on $K$ changes $\mu$ by $2$ as well.
\end{proof}

\end{remark}

\subsection{Basic fronts}  For $k \geq 1$, let $A_k \subset J^1(S^1)$ denote the Legendrian knot whose front diagram is given by identifying the ends of the $m$-stranded braid  $\sigma_1 \sigma_2 \cdots \sigma_{k-1}$, where strands are labeled from top to bottom, composition of braids is from left to right, and $\sigma_1,\ldots,\sigma_{k-1}$ represent the standard generators of the braid group $B_k$. That is, $A_k$ is the closure of the $k$-strand front \raisebox{-10pt}{\includegraphics[height=30pt]{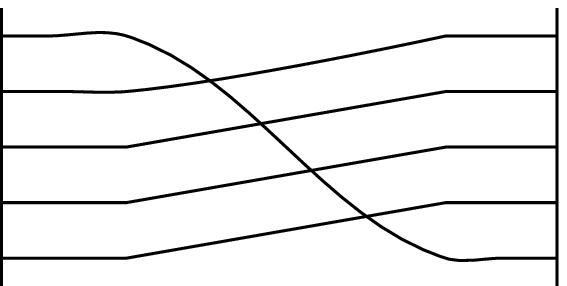}}, which winds $k$ times around the $S^1$ factor and has $k-1$ crossings.   We will refer to the $A_k$ as {\it basic fronts}.  Moreover, for $m \in \Z/\p$ we write $A_k^m$ for the basic front $A_k$
  with $\p$-graded Maslov potential identically equal to $m$.

Given front diagrams $L_1, L_2 \subset J^1(S^1)$, we define their product $L_1\cdot L_2$ by stacking $L_1$ above $L_2$.
We introduce some notation for $\p$-graded products of basic fronts (see Figure \ref{fig:BFEx} for an illustration).
Let $\lambda = (\lambda_1, \ldots, \lambda_\ell)$ be a (finite)
sequence of positive integers and ${\bf m} = (m_1, \ldots, m_{\ell})$
a sequence of elements of $\Z/\p$.  Writing $\Lambda = (\lambda, {\bf
  m})$ for a pair of such sequences, we define $A_\Lambda \subset
J^1(S^1)$ as the $\p$-graded Legendrian link
 $A_\Lambda = A^{m_1}_{\lambda_1} \cdots A^{m_\ell}_{\lambda_\ell}$.  Given such a pair $\Lambda$, we introduce a function
\[
{\bf n}_\Lambda : \Z/\p \to \Z_{\geq 0}, \quad \mbox{ where } \quad {\bf
  n}_\Lambda(k) = \sum_{i \text{ with } m_i=k} \lambda_i.
  \]
That is, ${\bf n}_\Lambda(k)$ denotes the total number of strands of a fixed $x$-coordinate with Maslov potential equal to $k$.

\begin{figure}
\centerline{\includegraphics[scale=.6]{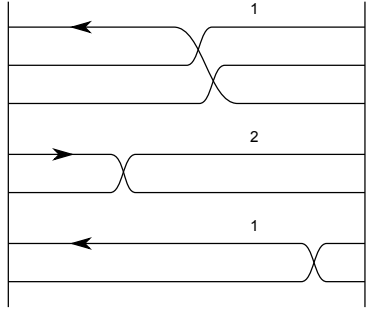}}
\caption{The product of basic fronts $A_\Lambda$ where $\Lambda = (\lambda, {\bf m})$ with $\lambda = (3, 2, 2)$ and ${\bf m}=(1,2,1)$.  The function ${\bf n}_{\Lambda}$ satisfies ${\bf n}_\Lambda(1) = 5, \, {\bf n}_\Lambda(2) = 2,$ and ${\bf n}_\Lambda(k) = 0$ for $k \neq 1,2$.}
\label{fig:BFEx}
\end{figure}

Later, we are able to reduce some questions about satellites with arbitrary pattern $L \subset J^1(S^1)$ to the case where $L$ is a product of basic fronts.

\subsection{Normal rulings}
\label{ssec:rulings}
The notion of normal ruling was developed independently in the works \cite{ChP,F}.   Let $L$ be a Legendrian knot in $\R^3$ or $J^1(S^1)$ whose front projection is generic in the sense described earlier in this section.  In addition, we now assume that all crossings and cusps have distinct $x$-coordinates.  This can be achieved after a small Legendrian isotopy.

In this section, we unify our notation by viewing $\R^3 \cong J^1(\R)$.  For $M = \R$ or $S^1$, we let $\pi : M\times \R \rightarrow M$ denote the projection $\pi(x,z) = x$, where the domain is viewed as the front projection of $J^1(M)$.
Let $\Sigma \subset M$ denote the projection of the set of cusp and crossing points of the front projection of $L$.  Furthermore, for any $x\in M$ let $L_x = \pi^{-1}(x)$.

\begin{definition} A continuous function $f$ from a subset $N \subset M$ to the front projection of $L \subset M\times \R$ is called a {\it section} if $\pi \circ f = \mathit{id}_N$.
\end{definition}

\begin{definition} \label{def:NR} A {\bf normal ruling} of the front projection of a link $L$ in $\R^3$ or $J^1(S^1)$ is a continuous involution, $\rho : L\setminus\pi^{-1}(\Sigma) \rightarrow L\setminus\pi^{-1}(\Sigma)$, $\rho^2 = \mathit{id}_{L\setminus\pi^{-1}(\Sigma)}$, satisfying the following:

\begin{itemize}
\item[(i)]  The involution $\rho$ is fixed point free.
\item[(ii)] The involution satisfies $\pi \circ \rho = \pi|_{L\setminus \pi^{-1}(\Sigma)}$ and therefore restricts to involutions $\rho_x : L_x \rightarrow L_x$ for each $x \in M \setminus \Sigma$.

For a component $N \subset M \setminus \Sigma$, the inverse image  of $N$ in $L$ is a union of ``strands'', $\displaystyle \pi^{-1}(N) = \bigsqcup S_i$, where each $S_i$ is mapped homeomorphically onto $N$ by $\pi$.  Due to the continuity condition, $\rho$ induces an involution of the collection of strands, and we say $\rho$ pairs $S_i$ with $S_j$ if $\rho(S_i) = S_j$.

\item[(iii)] In a neighborhood of a cusp point, the involution $\rho$ interchanges the upper and lower branch of the cusp.  On the remaining strands, the involution induced by $\rho_x$ should be the same on either side of the cusp.
\item[(iv)]  Strands meeting at a crossing should not be paired by $\rho$.
\item[(v)]  The involution $\rho$ extends continuously near crossings in the following sense.  If $N \subset M$ is such that $\pi^{-1}(N)$ contains a single crossing at $x$-coordinate $x_0$, then we can find sections $f_1, f_2, \ldots, f_N: N \rightarrow L \subset M\times \R$ such that every point of  $L \cap \pi^{-1}(N)$ is in the image of exactly one of the $f_i$ except for the crossing point which is in the image of exactly two of the $f_i$.  Moreover, these sections are preserved by the involution.  That is, for any $i = 1, \ldots, N$, there exists $j$ such that $\rho \circ f_i = f_j$ on $N \setminus \{x_0\}$.   It is clear that, except for their enumeration, the sections are uniquely determined by the involution $\rho$.

At the crossing point, there are two possibilities.  Either the two sections that meet at the crossing follow the diagram and cross in a transverse manner, or they each turn a corner at the crossing.  In the latter case, we refer to the crossing as a {\it switch} of $\rho$.  At a switch, one section covers the upper half of the crossing, and another covers the lower half.  Due to requirements (i) and (iv), each of these sections is paired by $\rho$ with a {\it companion strand} away from the crossing.

\item[(vi)] ({\it Normality condition}) Near a switch we can produce intervals on the vertical axis by connecting each switching strand with its companion strand.  These two intervals should either be disjoint or one should be contained in the other.  See Figure \ref{fig:NormC}.
\end{itemize}
\end{definition}

\begin{figure}
\centerline{\includegraphics[scale=1]{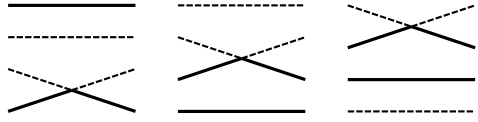}}
\caption{The normality condition.}
\label{fig:NormC}
\end{figure}

\begin{definition} \label{def:PGNR}
Suppose now that $(L, \mu)$ is a $\p$-graded Legendrian link.  We say that a ruling $\rho$ of $L$ is {\it $\p$-graded} with respect to $\mu$ if for
 $(x,z), (x,z')  \in L\setminus\pi_x^{-1}(\Sigma)$ with $z< z'$ and $\rho(x,z) = (x,z')$, we have $\mu(x,z') = \mu(x,z) +1$ mod $\p$.
\end{definition}

\begin{remark} \label{rem:NRDecomp} Alternatively, a normal ruling
 may be viewed as a global decomposition of the front diagram into pairs of sections, and we will make use of this perspective in our figures and proofs.  In the case of a link in $J^1(S^1)$, it is important here that we view the front diagram as a subset of $[0,1]\times \R$.  Then starting at $x=0$ or at the first left cusp of $L$ and working to the right, we can piece together the sections $f_i$ from condition (v).  This allows us to cover the front diagram of $L$ with a collection of sections with maximal domains of definition.  The involution then divides the sections into pairs $(P_i, Q_i)$ that begin and end at common cusps or possibly at common components of the boundary of $[0,1]\times \R$ in the case $L \subset J^1(S^1)$.

Note that in the case $L \subset J^1(S^1)$, a section that begins at $x=0$ does not necessarily have to end up at the same $z$-coordinate if it makes it all the way to $x=1$ without terminating at a cusp.  However, the involution $\rho$ is defined on the front diagram $L$ viewed as a subset of $S^1 \times \R$, so the overall division of points of $L$ into pairs at $x=0$ and $x=1$ should be the same.
\end{remark}

For Legendrian links in $J^1(S^1)$, it is appropriate for some purposes (see \cite{LR}) to relax the fixed point free condition of Definition \ref{def:NR} (i).

\begin{definition} Let $L \subset J^1(S^1)$.  A {\it generalized normal ruling} of $L$ is an involution $\rho$ satisfying the requirements of Definition \ref{def:NR} except for the following modifications.
\begin{itemize}
\item[(i)] The involution may have fixed points.

Near crossings, the locally defined sections $f_i$ are no longer uniquely determined by $\rho$ in the case where both of the crossing strands are fixed by $\rho$.  However, if at least one of the crossing strands is not fixed by $\rho$, then uniqueness still holds.  In particular, it is possible to have a switch where one of the switching strands has a companion strand and the other is a fixed point strand.  In this case, the normality condition is extended.

\item[(ii)] ({\it Generalized normality condition})  Near switches where one of the strands is fixed by $\rho$, the vertical interval connecting the non-fixed point strand to its companion strand should not intersect the other strand of the switch.  See Figure \ref{fig:GenNormC}.
\end{itemize}

\begin{figure}
\centerline{\includegraphics[scale=1]{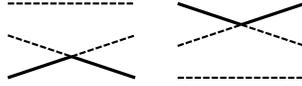}}
\caption{The generalized normality condition.}
\label{fig:GenNormC}
\end{figure}

Definition \ref{def:PGNR} carries over without change to provide a notion of $\p$-graded generalized normal ruling.
\end{definition}

\begin{remark}\label{rem:GNRDecomp}  A generalized normal ruling produces a decomposition of the front diagram of $L$ into pairs of sections $(P_i, Q_i)$ and a fixed point subset $F$ which does not contain cusps.  (Compare Remark \ref{rem:NRDecomp}.)  We make use of this perspective in our figures.
\end{remark}

Given a $\p$-graded Legendrian link $(L,\mu)$ in $J^1(\R)$ or $L\subset J^1(S^1)$, let $\mathcal{R}^\p(L, \mu)$ (resp. $\mathcal{GR}^\p(L, \mu)$)  denote the set of all normal rulings  (resp. generalized normal rulings) of $L$ which are $\p$-graded with respect to $\mu$.  Finally, we define the {\it $\p$-graded ruling polynomial}, $R_{(L,\mu)}^\p$, by
\[\displaystyle
R_{(L,\mu)}^\p(z) = \sum_{\rho \in \mathcal{R}^\p(L, \mu)} z^{j(\rho)} \quad \mbox{where} \quad  j(\rho) = \# \mbox{switches} - \# \mbox{right cusps}.
\]
Chekanov and Pushkar \cite{ChP} prove the following invariance result.

\begin{theorem}[\cite{ChP}] \label{thm:CP} If $(L,\mu_1)$ and $(L,\mu_2)$ are Legendrian isotopic as $\p$-graded links, then $R^\p_{(L_1, \mu_1)}(z) = R^\p_{(L_2, \mu_2)}(z)$.
\end{theorem}

\begin{remark} Note that ruling polynomials are unchanged by the addition of an overall constant to the Maslov potential.   In particular, if $K$ is a (connected) knot, then the $\p$-graded ruling polynomials are independent of the choice of $\mu$.  Moreover, if $K \subset \R^3$ is an oriented knot, then for any $\p$-graded $(L, \eta) \subset J^1(S^1)$ with $\p \,|\, 2 r(K)$, the polynomial $R^\p_{S(K,L)}(z)$ is a Legendrian isotopy invariant of $K$.
\end{remark}

When $\p=1$ or $2$ the ruling polynomials depend only on the underlying framed knot type of $L$.  This follows from:

\begin{theorem}[\cite{R}] \label{thm:Kauffman} For any Legendrian link $L \subset \R^3$, let $F_L, P_L \in \Z[a^{\pm 1}, z^{\pm 1}]$ denote the Kauffman and HOMFLY-PT link polynomials\footnote{We follow here the conventions of \cite{R} for the HOMFLY-PT and Kauffman polynomials.  However, our conventions for the power $j(\rho)$ appearing in the ruling polynomial differ by $1$ from \cite{R}.}.  Then the $1$-graded (resp. $2$-graded) ruling polynomial $R^1_L(z)$ (resp. $R^2_L(z)$) is equal to $z^{-1}$ times the coefficient of $a^{-\tb(L)-1}$ in $F_L$ (resp. $P_L$).
\end{theorem}

\begin{remark}  An analogous but more complicated result holds for links in $J^1(S^1)$.  See \cite{R2,LR}.
\end{remark}

\subsection{Chekanov--Eliashberg differential graded algebra} \label{sec:CEDGA}

In this subsection we recall the definition of the Chekanov--Eliashberg
DGA associated to a Legendrian knot in $\R^3$,
with some adjustments (related to base points and commutativity) to
adapt the standard treatment to the needs of this paper.

For the purposes of defining the DGA, it is more convenient to work in
the Lagrangian ($xy$) projection than in the front ($xz$) projection
used in the first part of this section. One can use an elementary
construction called resolution \cite{NgCLI} to produce a Lagrangian
projection from a front projection: diagrammatically, smooth out all
left cusps, and replace right cusps by a loop with a negative crossing.

Let $K$ be an oriented Legendrian knot with a base point $*$, generic
in the sense that the Lagrangian projection $\pi_{xy}(K)$ is immersed with only
transverse double points as singularities, and $\pi_{xy}(*)$ lies away
from the double points. Contact homology associates to $(K,*)$ a
differential graded algebra $(\A(K,*),\partial)$, the
\textit{Chekanov--Eliashberg algebra}, as we now briefly recall; see
e.g. \cite{Ch,ENS} for more details.

\begin{definition}
Label the crossings (double points) of $\pi_{xy}(K)$ by
$a_1,\ldots,a_n$. The algebra $\A(K,*)$ is the associative, noncommutative unital
algebra over $\Z/2$ generated by
\[
a_1,\ldots,a_n,t,t^{-1}
\]
with no relations besides $t\cdot t^{-1} = t^{-1}\cdot t = 1$.
\end{definition}

The algebra $\A(K,*)$ is generated as a $(\Z/2)$-vector space by words of the form
\[
t^{\alpha_0} a_{i_0} t^{\alpha_1} a_{i_1} \cdots a_{i_k} t^{\alpha_k}
\]
(including the empty word, which serves as the identity element $1$),
with multiplication given by concatenation. We note that this
definition of $\A(K,*)$ is slightly different from the corresponding
definition in \cite{ENS}, even accounting for the fact that we work
over $\Z/2$ and not $\Z$: the algebra considered in \cite{ENS} is the
quotient of ours by allowing powers of $t$ to commute with
the $a_i$'s. This construction follows \cite{NgLSFT}; see
\cite[section~2.3.2]{EENS} for further discussion.

We next give $\A(K,\ast)$ a $\Z/(2r(K))$-grading (and thus a $(\Z/\p)$-grading
for any $\p \,|\,2r(K)$). If $a_i$ is a crossing of $\pi_{xy}(K)$, let
$\gamma_i \subset \R^2$ be a path along $\pi_{xy}(K)$ beginning at the
overcrossing and ending at the undercrossing of $a_i$. Let
$r(\gamma_i)$ be the (non-integral) number of counterclockwise
revolutions made by the unit tangent
vector to the path $\gamma_i$ from beginning to end, and define
$|a_i| = \lfloor 2r(\gamma_i) \rfloor$. Note that the
grading is well-defined, independent of the choice of $\gamma_i$,
modulo $2r(K)$. Extend the grading to all of $\A(K,\ast)$ by setting $|t| =
|t^{-1}| = 0$ and extending in the usual way (the degree of a product
is the sum of the degrees).

If $\pi_{xy}(K)$ is a
resolution of a front diagram endowed with a Maslov potential, then the degree
of crossing $a_i$ is the difference of the Maslov potentials
associated to the strands passing through $a_i$. We can also use this to define a grading on the DGA in the more general case where $K$ is a multi-component link.

We now come to the differential on $\A$. Attach signs to each corner
at every crossing in $\pi_{xy}(K)$ as
depicted: \raisebox{-11pt}{\includegraphics[height=33pt]{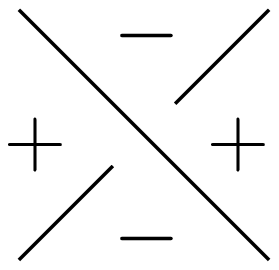}}.
Let $D^2$ denote the closed
disk. For $\ell\geq 0$, write $D^2_\ell = D^2\setminus\{r,s_1,\ldots,s_\ell\}$, where
$r,s_1,\ldots,s_\ell$ are points on $\partial D^2$ appearing in order as
we traverse the boundary counterclockwise.

\begin{definition}
Let $a,b_1,\ldots,b_\ell$ be crossings in the Lagrangian projection
$\pi_{xy}(K)$ of a Legendrian knot $K \subset \R^3$. Define
$\Delta(a;b_1,\ldots,b_\ell)$ to be the space of all
orientation-preserving immersions $f :\thinspace (D^2_\ell,\partial D^2_\ell)
\to (\R^2,\pi_{xy}(K))$, up to reparametrization, such that:
\begin{itemize}
\item
 $f$ sends
the boundary punctures of $D^2_\ell$ to the crossings of $\pi_{xy}(K)$
\item
$f$ sends a neighborhood of $r$ to a corner at $a$ labeled by a $+$
\item
$f$ sends a neighborhood of each $s_i$ to a corner at $b_i$ labeled by
a $-$.
\end{itemize}
\end{definition}

For $f \in \Delta(a;b_1,\ldots,b_\ell)$, the image of $\partial D^2_\ell$
maps to a union of $\ell+1$ paths $\gamma_0,\ldots,\gamma_\ell \subset
\pi_{xy}(K)$, where each path begins and ends at a crossing, and
$\gamma_0$ goes from $a$ to $b_1$, $\gamma_i$ from $b_i$ to $b_{i+1}$
for $1\leq i\leq \ell-1$, and $\gamma_\ell$ from $b_\ell$ to $a$. For each
of these paths $\gamma_i$, we can associate a monomial $w(\gamma_i)
\in \Z/2[t^{\pm 1}]$ by $w(\gamma_i) = t^{\alpha_i}$, where $\alpha_i$
is the number of times $\gamma_i$ passes through the base point
$\ast$, counted with sign according to the orientation of $K$.
Finally, we can associate a monomial $w(f) \in \A(K,*)$ to $f$, as
follows:
\[
w(f) = w(\gamma_0) b_1 w(\gamma_1) b_2 \cdots b_\ell w(\gamma_\ell).
\]

\begin{definition}
Let $a$ be a crossing of $K$.
\label{def:diff}
The differential $\partial(a)$ is
defined by:
\[
\partial(a) = \sum_{f\in\Delta(a;b_1,\ldots,b_\ell)} w(f),
\]
where the sum is over all $\ell\geq 0$ and all choices of crossings
$b_1,\ldots,b_\ell$ such that $\Delta(a;b_1,\ldots,b_\ell)$ is nonempty.
\end{definition}

We can extend the map $\partial$ to all of $\A(K,*)$ by
setting $\partial(t) = \partial(t^{-1}) = 0$ and imposing the Leibniz
rule.

\begin{theorem}[\cite{Ch,ENS}] \label{thm:DGAinv}
The map $\partial :\thinspace \A(K,*) \to \A(K,*)$ lowers degree by
$1$ and is a differential: $\partial^2 = 0$. Up to stable tame
isomorphism, the differential graded algebra $(\A(K,*),\partial)$ is
an invariant of $K$ under Legendrian isotopy (and choice of base point).
\end{theorem}

Here ``stable tame isomorphism'' is an equivalence relation of
differential graded algebras that in particular fixes $t$ and preserves isomorphism
type of the homology $H_*(\A(K,*),\partial)$; see \cite{Ch}, or
\cite{EENS} for a definition in our setting. For the purposes of this
paper, this relation may be treated as a black box.

In Section~\ref{sec:reps}, we will need a slight generalization of the
above notion of the Chekanov--Eliashberg DGA, to the setting where we
have multiple base points $*_1,\ldots,*_k$ on $K$. As before, we
assume that in the $xy$ projection, no base point coincides with a
crossing; we also assume that the base points are cyclically ordered
along $K$, i.e., $*_1,\ldots,*_k$ are encountered in that order as we
traverse the
knot in the direction of its orientation.

Given this data, we define the algebra $\A(K,*_1,\ldots,*_k)$ to be
the noncommutative unital algebra over $\Z/2$ generated by crossings
$a_1,\ldots,a_n$, along with $2k$ additional generators $t_1^{\pm
  1},\ldots,t_k^{\pm 1}$, with no relations besides $t_i \cdot
t_i^{-1} = t_i^{-1} \cdot t_i = 1$ for all $i$. (Note in particular
that the $t_i$'s do not commute with the $a$'s, or indeed with each other.)
We give $\A(K,*_1,\ldots,*_k)$ a $\Z/(2r(K))$-grading as before, with
$|t_i| = |t_i^{-1}| = 0$ for all $i$.

We can define a differential $\partial$ on $\A(K,*_1,\ldots,*_k)$
analogously to Definition~\ref{def:diff}. Note that in the presence of
multiple base points, the monomial $w(\gamma)$ associated to a path
$\gamma$ in $\pi_{xy}(K)$ can involve any or all of $t_1^{\pm
  1},\ldots,t_k^{\pm 1}$: it is the product $t_{i_1}^{\pm 1} \cdots
t_{i_l}^{\pm 1}$, where $\gamma$ passes through
$*_{i_1},\ldots,*_{i_l}$ in succession, and the signs depend on
whether the orientation of $\gamma$ agrees or disagrees with the
orientation of $K$ as $\gamma$ passes through the base point.

The DGA $(\A(K,*_1,\ldots,*_k),\partial)$ depends only minimally on
the choice of base points, and indeed contains no more information
than the single-base-pointed DGA $(\A(K,*),\partial)$. More precisely,
we have the following results.

\begin{theorem}\label{prop:multi1}  Let $*_1, \ldots, *_k$ and
  $*_1', \ldots, *_k'$ denote two collections of base points on $K$,
  each of which is cyclically ordered along $K$. Let
 $(\A(K, *_1, \ldots, *_k),\d)$ and $(\A(K, *_1', \ldots, *_k'),\d')$
 denote the corresponding multi-pointed DGAs.
Then there is a DGA isomorphism $\phi :\thinspace (\A(K, *_1, \ldots,
*_k),\d) \rightarrow
(\A(K, *_1', \ldots, *_k'),\d')$ such that $\phi(t_i) = t_i$ for all $i$.
\end{theorem}

\begin{proof}
It suffices to establish the result when $(*_1,\ldots,*_k)$ and
$(*_1',\ldots,*_k')$ are identical except that for some $i$, $*_i'$ is
the result of sliding $*_i$ across a crossing of
$\pi_{xy}(K)$. Suppose then that $*_i$ and $*_i'$ lie on opposite
sides of a crossing $a_l$, with the orientation of $K$ pointing from
$*_i$ to $*_i'$. We first consider the case where the strand
containing $*_i$ and $*_i'$ is the overstrand at $a_l$. In this case,
if $f$ is a disk with a positive corner at $a_l$ and $w(f),w'(f)$ are the words
associated to $f$ in $\A(K,*_1,\ldots,*_j,\ldots,*_k)$ and
$\A(K,*_1,\ldots,*_j',\ldots,*_k)$ respectively, then $w'(f) =
t_i w(f)$. Furthermore, if $f$ is a disk with a negative corner at
$a_l$ , then $w'(f)$ is the result of replacing $a_l$ by $t_i^{-1} a_l$ in
$w(f)$. It follows that the map $\phi$ defined by
$\phi(a_l) = t_i^{-1} a_l$, $\phi(a_j) = a_j$ for $j\neq l$, and $\phi(t_j)
= t_j$ for all $j$ satisfies $\phi\circ\partial = \partial'\circ\phi$.

If the strand containing $*_i$ and $*_i'$ is the understrand at $a_l$,
a similar argument shows that the map $\phi$ defined by $\phi(a_l) =
a_l t_i$, $\phi(a_j) = a_j$ for $j\neq l$, and $\phi(t_j) = t_j$ for
all $j$ satisfies $\phi\circ\partial = \partial'\circ\phi$.
\end{proof}

\begin{theorem}\label{prop:multi2}
Let $*_1,\ldots,*_k$ be a cyclically ordered collection of base points
along $K$, and let $*$ be a single base point on $K$. Then there is a
DGA homomorphism $\phi :\thinspace (\A(K,*),\d) \to
(\A(K,*_1,\ldots,*_k),\d)$ such that $\phi\circ\d = \d\circ\phi$ and
$\phi(t) = t_1\cdots t_k$.
\end{theorem}

\begin{proof}
By Theorem~\ref{prop:multi1}, we may assume that $*_1,\ldots,*_k$
all lie in a small neighborhood of $*$. In this case, the map $\phi$ defined
by $\phi(a_i) = a_i$ for all crossings $a_i$ and $\phi(t) = t_1\cdots
t_k$ is the desired homomorphism.
\end{proof}

\subsection{Augmentations and representations of the DGA} \label{ssec:rep}

Here we discuss representations of the DGA introduced in the previous
subsection. We begin with augmentations, which can be viewed as
$1$-dimensional representations.

\begin{definition}
Let $\p \,|\,2r(K)$.
A \textit{$\p$-graded augmentation} of $(\A(K,\ast),\partial)$ is an algebra
map $\epsilon :\thinspace \A(K,\ast)\to\Z/2$ such that:
\begin{itemize}
\item
$\epsilon(1) = \epsilon(t) = \epsilon(t^{-1}) = 1$;
\item$\epsilon \circ \d = 0$;
\item $\epsilon(a) = 0$ if $a\in\A$ with $|a| \not\equiv 0 \pmod{\p}$.
\end{itemize}
\end{definition}

Stable tame isomorphism (discussed briefly in the previous subsection)
preserves the existence and nonexistence of $\p$-graded
augmentations. Theorem~\ref{thm:DGAinv} then immediately implies the following.

\begin{theorem}
If $K$ and $K'$ are Legendrian isotopic knots with base points $*$ and
$*'$, then for any $\p \,|\,2r(K)$, the Chekanov--Eliashberg DGA
$(\A(K,*),\partial)$ has a $\p$-graded augmentation if and only if
$(\A(K',*'),\d)$ does.
\end{theorem}

There is a well-known correspondence between augmentations and
rulings:

\begin{theorem}[\cite{F,FI,Sabloff}] \label{thm:FIS}
Let $K$ be a Legendrian knot in $\R^3$, and let $\p \, | \,
2r(K)$. Then the front projection of $K$ has a $\p$-graded ruling if
and only if the DGA $(\A(K,*),\d)$ has a $\p$-graded augmentation.
\end{theorem}

We need the following more precise statement.  Recall that a Legendrian link in $\R^3$ is said to be in {\it plat position} if all right cusps have the same $x$-coordinate as do all left cusps.

\begin{theorem}[\cite{Sabloff, Henry2011}]  \label{prop:bg} Let $K$ be a $\p$-graded Legendrian link in
  $\R^3$ with front diagram in plat position, and denote by $(\A,\d)$
  the
Chekanov--Eliashberg DGA associated with the resolution of $K$.
\begin{itemize}
\item[(i)] Given any $\p$-graded augmentation of $(\A,\d)$, there exists a $\p$-graded normal ruling $\rho$ of $K$ so that the first switch of $\rho$ occurs at or to the right of the first augmented crossing.

\item[(ii)] For any $\p$-graded normal ruling $\rho$ of $K$ ,there is a $\p$-graded augmentation of $(\A,\d)$ so that the first augmented crossing agrees with the first switch
     of $\rho$.
\end{itemize}
\end{theorem}

\begin{proof}  Statement (i) follows from an algorithm in \cite[section 3.3]{Sabloff} that assigns a normal ruling to an augmentation of a plat position front diagram.  This algorithm is also presented in \cite{NgSabloff}, and it is easy to see that the first switch of the normal ruling must be to the right of the first augmented crossing.

For (ii), we cite work of Henry \cite{Henry2011}.  The main objects of study in \cite{Henry2011} are ``Morse complex sequences,'' which consist of sequences of chain complexes assigned  to a front diagram of a Legendrian link in $\R^3$.  In \cite[section 6.5]{Henry2011}, two different standard forms for a Morse complex sequence (MCS) are introduced, the $S\bar{R}$-form and the $A$-form.  These standard forms are related to normal rulings and augmentations respectively.

Given a normal ruling $\rho$ as in (ii), we consider the $S\bar{R}$-form MCS associated with $\rho$ where none of the returns have handleslides.    Theorem 6.20 of \cite{Henry2011} shows that this MCS may be transformed into an $A$-form MCS by an algorithm that sweeps handleslide marks from left to right.  In particular, the leftmost handleslide in this $A$-form MCS will be located directly to the left of the first switch of $\rho$.  According to Corollary 6.29 of \cite{Henry2011}, this $A$-form MCS corresponds to an augmentation of the Chekanov--Eliashberg algebra of the resolution of $K$ where a crossing is augmented if and only if there is a handleslide immediately to the left of the crossing.  This augmentation has the desired form.
\end{proof}

We now generalize our discussion from augmentations to representations
of the DGA.
Suppose $V$ is a finite-dimensional vector space over $\Z/2$ with a $\Z/\p$-grading, $\displaystyle V = \bigoplus_{k \in \Z/\p} V_k$.  Then $\displaystyle \mathit{End}(V) = \bigoplus_{i,j} \mathit{Hom}_{\Z/2}(V_i, V_j)$ is a $\Z/\p$-graded algebra where we take each  $\mathit{Hom}_{\Z/2}(V_i, V_j)$ to be homogeneous of degree\footnote{ Note that with this convention, our definition of degree is the \textit{negative} of the standard grading for graded linear maps.}   $i -j \in \Z/\p$.

\begin{definition} \label{def:rep}
A \textit{$\p$-graded representation} of $(\A,\d)$ is a $\p$-graded vector space $V$ over $\Z/2$
along with a DGA map from $(\A,\d)$ to $(\End(V),0)$, i.e., a grading-preserving algebra map
$f :\thinspace\A\rightarrow \End(V)$ satisfying $f(1) = \mathit{id}_V$ and $f \circ \d = 0$.  The {\it graded dimension} of the representation, $\dim(V)$, is the function
${\bf n}: \Z/\p \rightarrow \Z_{\geq 0}$ defined by ${\bf n}(k) = \dim V_k.$
\end{definition}

Note that a representation of $(\A(K,*),\d)$ does not need to send $t$
to $\mathit{id}_V$, but merely to an invertible map on $V$.

\section{Normal Rulings of Legendrian Satellites}
\label{sec:rulings}

We begin this section by establishing some basic properties of normal rulings of Legendrian satellites. The definition of normal ruling requires working with a front projection of $S(K,L)$ with the property that crossings have
distinct $x$-coordinates.  To achieve this, we assume that $(K,*)$ and $L$ are in general position, and then apply a planar isotopy to perturb the $x$-coordinates of the crossings of the front diagram for $S(K,L)$ described in Section \ref{sub:satellite}.  The precise order that the crossings end up in will not be relevant for our arguments.  In this section, we continue to use the convention of labeling the parallel translates of a strand of $K$ appearing in $S(K,L)$ from $1$ to $n$ with decreasing $z$-coordinate.

We introduce some terminology associated with a normal ruling $\rho$ of $S(K,L)$.  Outside of neighborhoods of cusps, crossings, and the base point we can assign an involution, $\rho_T$, of $\{1, \ldots, n\}$ to each point $ k_0 \in K$ according to:
\begin{itemize}
\item[(i)] Let $\rho_T(i) = j$ if, of the $n$ parallel copies of $K$ in $S(K,L)$ that correspond to the point $k_0$, the ruling $\rho$ pairs the $i$-th strand with the $j$-th strands, and

\item[(ii)] let $\rho_T(i) = i$ if $\rho$ pairs the $i$-th strand of $S(K,L)$ at $k_0$ with a strand of $S(K,L)$ corresponding to a point other than $k_0$ in $K$.
\end{itemize}
We will refer to $\rho_T$ as the {\it thin part} of $\rho$ at $k_0$.  In addition, we refer to strands of $S(K,L)$ that correspond to the same strand of $K$ and are paired by $\rho$ as a pair of {\it thin strands} of $\rho$.

In combination, the following two lemmas show that the thin part of $\rho$ is independent of $k_0 \in K$.

\begin{lemma}[Crossing Lemma]  A pair of thin strands cannot be involved in any switches at crossings of $S(K,L)$ that correspond to a crossing of $K$.
\end{lemma}

\begin{proof}   Let $Y$ denote the subset of the front diagram $S(K,L)$ corresponding to a neighborhood of a crossing $q$ of $K$.  We prove the more general statement:

\noindent{\it Let $P$ and $Q$ be a pair of companion paths of the ruling.  If both, $P$ and $Q$ pass through the region $Y$, then neither of them can  switch in $Y$.}

The proof is by induction on, $M$, the number of strands lying in the vertical interval between $P$ and $Q$ at such a switch.  The base case of $M = 1$ is prohibited by the normality condition.  For the inductive step, by symmetry we may assume that $P$ lies above $Q$ and such a switch occurs along $P$.

\begin{itemize}
\item[Case 1:]  The corner along $P$ at the switch points toward $Q$. Then by the normality condition, the other path at the switch must have its companion path between $P$ and $Q$.  The inductive hypothesis then provides a contradiction.

\item[Case 2:]  The corner along $P$ at the switch points away from $Q$. Then heading in the appropriate direction, (left or right, depending on the slope of $Q$ at the switch), $P$ and $Q$ will be on course to intersect within $Y$.  This is prohibited, so there must be another switch along either $P$ or $Q$.  This next switch will be of the type covered by Case 1, and, prior to this switch, the number of strands between $P$ and $Q$ can only decrease since they are angled toward one another in this direction.  Thus, the inductive hypothesis may be applied.
\end{itemize}
\end{proof}

\begin{lemma}[Cusp Lemma]  The thin part of $\rho$ does not change when passing a cusp.  Furthermore, at those crossings of $S(K,L)$ corresponding to a particular cusp of $K$ the crossing between the $i$-th strand and the $j$-th strand is a switch if and only if $\rho_T(i)= j$ before and after the cusp.
\end{lemma}

\begin{proof}  By symmetry we may consider the case of a left cusp of $K$.  Let $C_1, \ldots, C_n$ denote the corresponding left cusps of $S(K,L)$ numbered from top to bottom.  The ruling $\rho$ provides a pair of paths $P_i$ and $Q_i$ emanating from each of the $C_i$ where we assume $P_i$ to have larger $z$-coordinate than $Q_i$.  Although, it need not literally be the case, we will refer to those strands of $S(K,L)$ which correspond to the upper (resp. lower) branches of cusps as positive (resp. negative) sloped.  (Compare Figure \ref{fig:ThinCusp}.)

\noindent {\bf Claim:}  {\it Among those crossings near the cusp of $K$, none of the $P_i$ (resp. $Q_i$) can have a switch where the slope increases (resp. decreases).}

This is proved by induction on, $M$, the number of strands lying between $P_i$ and $Q_i$ at the offending switch $S$.  The base case $M=1$ would violate the normality condition.  By symmetry we may assume $S$ is a switch along $P_i$.  The normality condition forces that there is some $j \neq i$ such that (i) the lower half of $S$ is an arc along $P_j$, and (ii) $Q_j$ lies above $Q_i$ at $x(S)$.

\begin{itemize}
\item[Case 1:]  $Q_j$ is positively sloped at $x(S)$. Then immediately to the right of $S$, $P_j$ is negatively sloped and $Q_j$ is positively sloped.  To the right of $S$, one of these paths must switch or they will intersect before leaving the nearby collection of crossings.  However, as we move to the right of $S$ the number of strands between $P_j$ and $Q_j$ can only decrease, therefore this contradicts the inductive hypothesis.

\item[Case 2:]  $Q_j$ is negatively sloped at $x(S)$. This time follow $P_j$ and $Q_j$ to the left of $S$.  Since they lie between $P_i$ and $Q_i$ at $x(S)$, and because of the way they are sloped, one of $P_j$ and $Q_j$ must switch in order for them to meet at the cusp $C_j$.  However, the number of strands between $P_j$ and $Q_j$ decreases when moving to the left, and thus the inductive hypothesis applies to provide a contradiction at the first such switch.
\end{itemize}

Finally, we deduce the result from the Claim.  Near the cusp, each pair of paths $P_i$ and $Q_i$ can have at most one switch combined.  Indeed, the claim forbids both $P_i$ and $Q_i$ from individually having more than one switch, and if both $P_i$ and $Q_i$ have a single switch then they will intersect before leaving the collection of crossings near the cusp.  Now, in the case that there is no switch along $P_i$ or $Q_i$, the $i$-th strand will be a fixed point of $\rho_T$ along either branch of the cusp.  Any switches that do occur must have upper strand $Q_i$ and lower strand $P_j$ with $i <j$, and such a switch results in $\rho_T(i) =j$ before and after the crossing.  Conversely, if $\rho_T(i) =j$ before or after the crossing, then a switch of this type is necessary in order for the corresponding thin strands to meet at a common cusp.
\end{proof}

\subsection{Reduced normal rulings}

\begin{definition}  For a non-empty pattern $L \subset J^1(S^1)$, we say that a normal ruling $\rho \in \mathcal{R}^\p(S(K,L))$ is {\it reduced} if
\begin{itemize}
\item[(i)] the front projection of $L$ intersects $x =0$ and
\item[(ii)] there are no switches at the crossings of $S(K,L)$ which arise from left cusps of $K$.
\end{itemize}
Denote by $\widetilde{\mathcal{R}}^\p(K,L)$ the set of reduced $\p$-graded normal rulings of $S(K,L)$ and by $\widetilde{R}^\p_{S(K,L)}(z)$ the corresponding {\it reduced ruling polynomial}.
When $L = \emptyset$ we make the convention that the empty ruling is reduced so that $\widetilde{\mathcal{R}}^\p(K,\emptyset)$ contains a single element.
\end{definition}

Any normal ruling of $S(K,L)$ corresponds to a generalized normal ruling $\tau$ of $L$ along with a reduced normal ruling of a certain satellite associated to $\tau$. More precisely, we have the following.

\begin{theorem} \label{thm:bijection}  Assume that $(K,*) \subset \R^3$ and $L \subset J^1(S^1)$ have front diagrams in general position.   There is a
  bijection
  $\Phi: T \stackrel{\cong}{\rightarrow}\mathcal{R}^1(S(K,L))$ where
  $T$ is the set of ordered pairs $(\tau, \sigma)$ satisfying $\tau
  \in \mathcal{GR}^1(L)$ and $\sigma \in
  \widetilde{\mathcal{R}}^1(K, L^\tau)$.  Here, $L^\tau \subset
  J^1(S^1)$ denotes the link whose front diagram corresponds to the portion   of $L$ that is fixed by $\tau$.  Furthermore, $j(\Phi(\tau,\sigma)) = j(\tau) +j(\sigma)$, and $\Phi(\tau,\sigma)$ is $\p$-graded if and only if $\tau$ and $\sigma$ are.
\end{theorem}

\begin{proof}

Given a pair $(\tau, \sigma) \in T$, $\Phi(\tau, \sigma)$ is constructed as follows.

\begin{itemize}
\item[Step 1.]   Extend $\tau$ to a partial ruling, $\widetilde{\tau}$ of $S(K,L)$ with some fixed point strands:

This is done by letting $\widetilde{\tau} = \tau$ along the subset $L \subset S(K,L)$ and then extending so that away from cusps the thin part of $\widetilde{\tau}$ agrees with the involution $\tau$ at the boundary of the front projection of $L$ in $[0,1]\times \R$.
That is, except near cusps, the $i$-th and $j$-th parallel copies of a strand of $K$ in $S(K,L)$ are paired by $\widetilde{\tau}$ if and only if on the vertical line $x=0$, $\tau(i)= j$.  All other strands of $S(K,L)$ are fixed point strands for $\widetilde{\tau}$.  Finally, to piece $\widetilde{\tau}$  together, for each such $i$ and $j$ we add one switch at every cusp.  Assuming $i < j$, these switches occur where the $i$-th strand corresponding to the lower branch of the cusp passes over the $j$-th strand.  See Figure \ref{fig:ThinCusp}.

\begin{figure}
\centerline{\includegraphics[scale=.5]{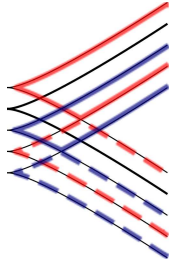}}
\caption{The thin part of a ruling of $S(K,L)$ near cusps.}
\label{fig:ThinCusp}
\end{figure}

\item[Step 2.]  Extend $\widetilde{\tau}$ to $\Phi(\tau, \sigma)$:

The fixed point strands of $\widetilde{\tau}$ (after smoothing near any relevant switches within $L$) form a front diagram which is combinatorially the same as $S(K,L^\tau)$.  We extend by requiring that the restriction to $S(K,L^\tau)$ is the normal ruling $\sigma$.  See Figure \ref{fig:BijEx}.
\end{itemize}

The switches of $\Phi(\tau, \sigma)$ can be divided into three disjoint types:  (A) Switches of $\tau$, (B) switches of $\sigma$, and (C) switches near cusps added in Step 1.

\begin{figure}
\[
\begin{array}{ccc}
\includegraphics[scale=.6]{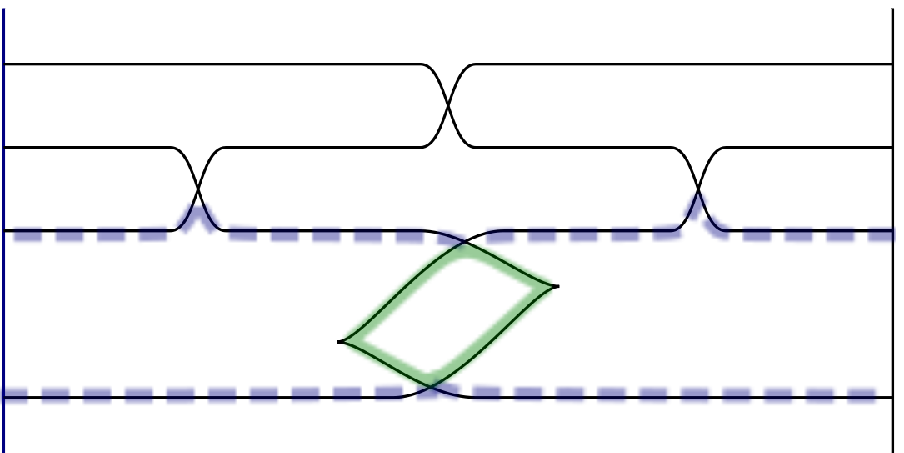} & \quad &\raisebox{2.1 ex}{\includegraphics[scale=.3]{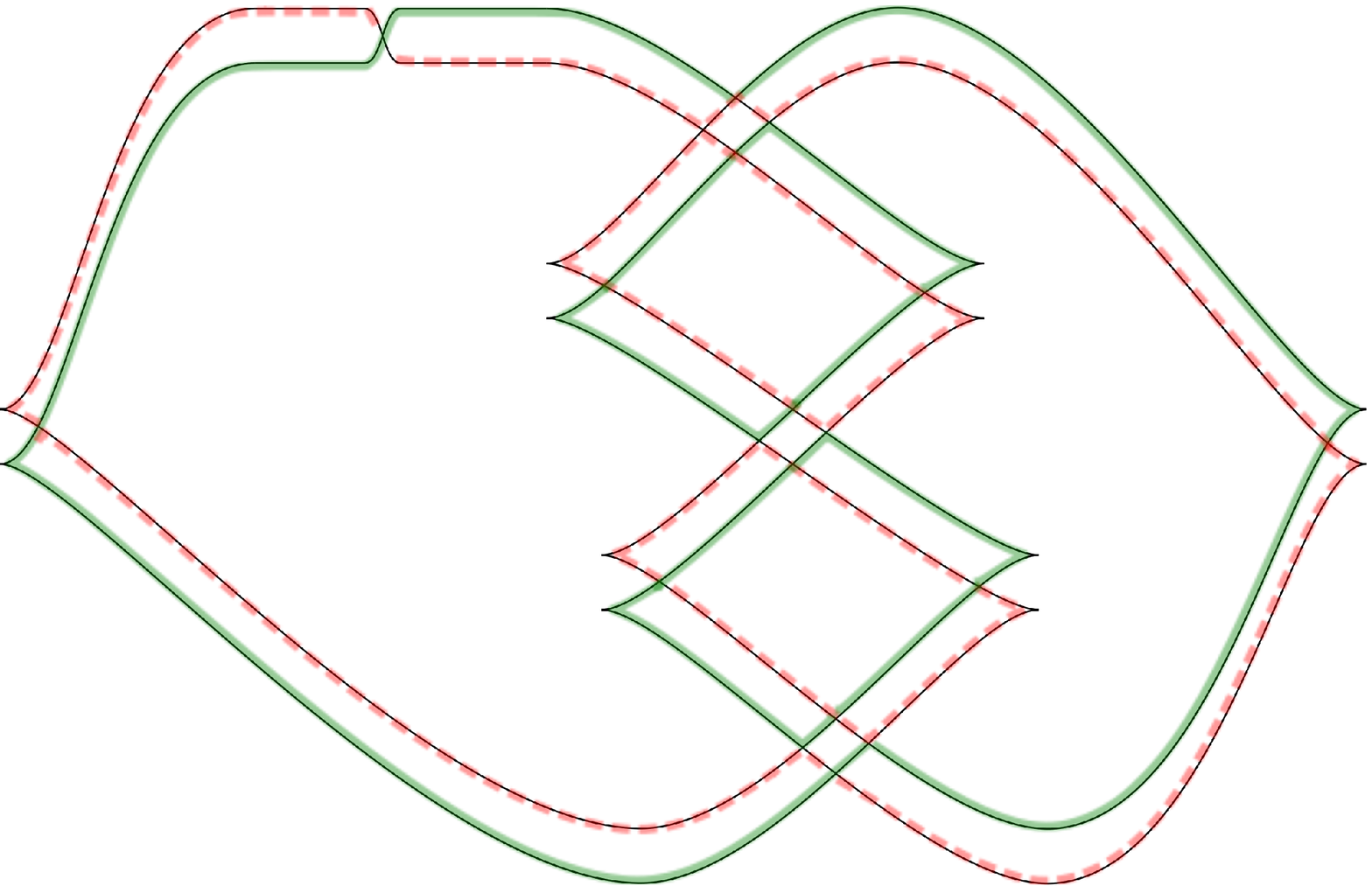}} \\
\tau & \quad & \sigma
\end{array}
\]
\centerline{\includegraphics[scale=.4]{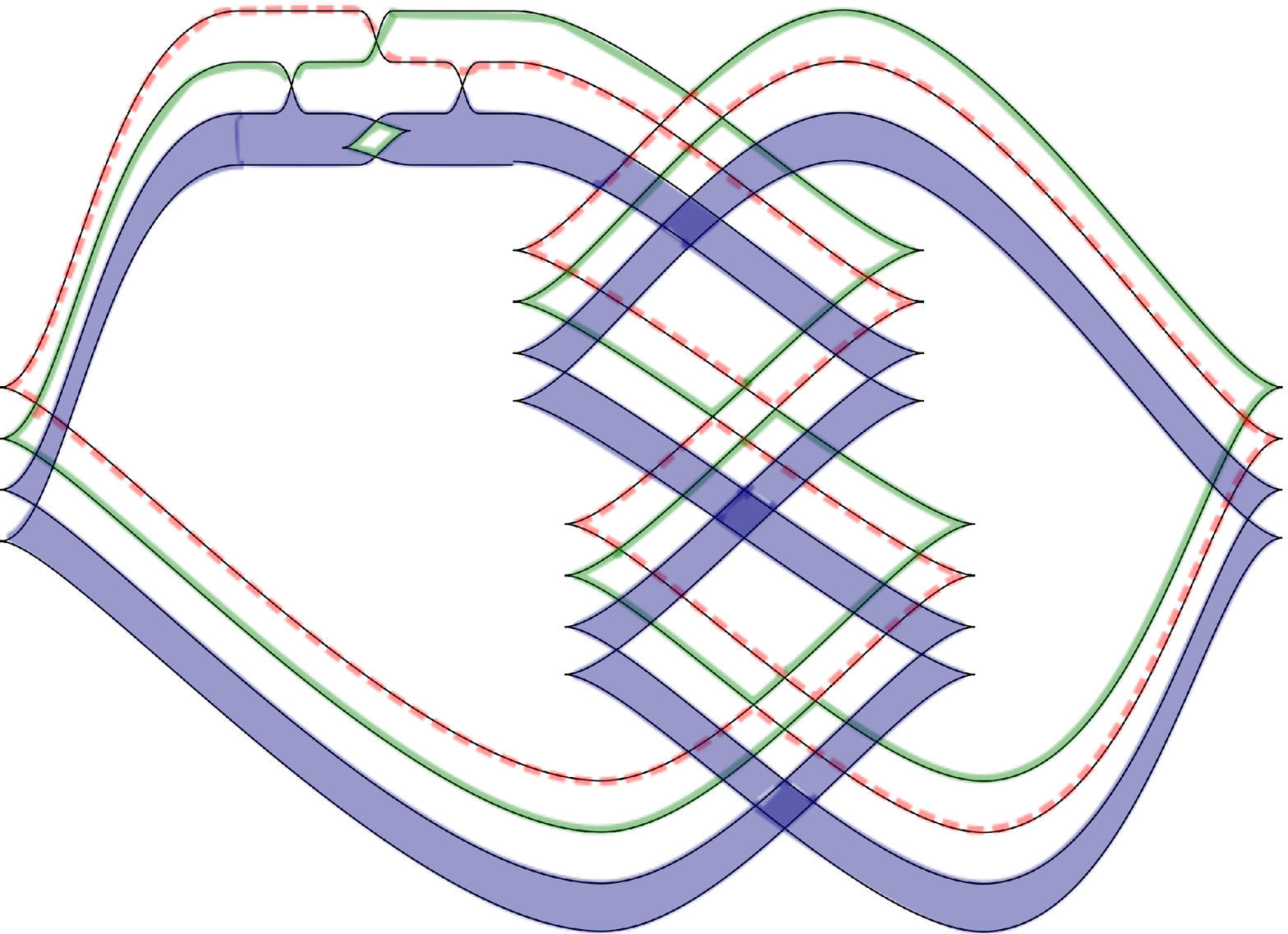}}
\centerline{$\Phi(\tau, \sigma)$}
\caption{The bijection from Theorem \ref{thm:bijection}.  In the pictured example, $K$ is a left-handed trefoil (topologically the mirror of the trefoil pictured in Figure \ref{fig:SatEx}).  The fixed point subset $L^\tau$ is the basic front $A_2$.}
\label{fig:BijEx}
\end{figure}

The normality condition is easily verified for switches of type (B) and (C).  For switches of type (A) we need to consider two subcases.  If neither of the involved strands are fixed point strands of $\tau$, then the normality condition follows since it holds in $\tau$.  If one of the switching strands, say $P_0$, is a fixed point strand of $\tau$, then the normality condition for the ruling $\Phi(\tau, \sigma)$ follows from the generalized normality condition for $\tau$ provided that we know the companion strand of $P_0$ in $\Phi(\tau, \sigma)$ lies outside of $L \subset S(K,L)$.  This is true for the following reason.  The fixed point subset, $L^\tau$, cannot contain cusps, so if their are a pair of thin strands of $\sigma$ within the subset $L^\tau \subset S(K, L^\tau)$, then we will continue to have a pair of thin strands immediately to the left of $L$.  Following this strand of $K$ to the left the Crossing Lemma implies that we continue to have a pair of thin strands until we reach a left cusp of $K$.  Finally, the Cusp Lemma then contradicts the assumption that $\sigma$ is a \underline{reduced} ruling of $S(K, L^\tau)$.

Next, we verify that $j(\Phi(\tau, \sigma)) = j(\tau) + j(\sigma)$.  The cusps of $S(K,L)$ not belonging to the subsets $L$ and $S(K, L^\tau)$ are in two-to-one correspondence with switches of type (C).  Recall that the negative term of $j$ counts $\frac{1}{2}$ the total number of cusps.  Thus, in the computation of $j(\Phi(\tau, \sigma))$ switches of type (C) precisely cancel these unclaimed cusps and we are left with $j(\tau)+ j(\sigma)$.  (Here, we used that $L^\tau$ does not contain cusps, so that none of the cusps of $L$ are double counted in $S(K, L^\tau)$.)

Now we check that $\Phi(\tau, \sigma)$ is $\p$-graded if $\tau$ and $\sigma$ are.  In the construction the only paired paths of $\Phi(\tau, \sigma)$ that do not belong to either $L$ or $S(K, L^\tau)$ are the thin pairs arising from Step 1.  Clearly, immediately next to $L$ the requirement of Definition \ref{def:PGNR} continues to hold.  Following along $K$, it is enough to verify that the condition continues to hold for thin pairs after passing a cusp.  This is immediate.  Near cusps, thin pairs belonging to the upper and lower branch of the cusp are in correspondence, and the Maslov potentials differ by $\pm 1$.

Finally, we show that $\Phi$ is onto.  Let $\rho \in \mathcal{R}^1(S(K,L))$ be arbitrary.  The restriction of $\rho$ to $L \subset S(K,L)$ produces a generalized normal ruling by making the convention that strands of $L$ mapped outside of $L$ by $\rho$ are fixed point strands of $\tau$.  The cusp and crossing lemmas imply that the thin part of $\rho$ is constant and that, precisely as in the construction from Step 1, the only switches outside of $L$ involving thin strands are near cusps of $K$.  Moreover, the complement of the thin strands of $S(K,L)$ will be precisely $S(K, L^\tau)$ and restricting $\rho$ to this subset produces $\sigma$ such that $\Phi(\tau, \sigma) = \rho$.
\end{proof}

\begin{corollary} \label{cor:estimate}  For any Legendrians $K \subset \R^3$ and $L \subset J^1(S^1)$ with $\Z/\p$-valued Maslov potential, we have
\[
R_{S(K,L)}^\p(z) \geq R_L^\p(z),
\]
where $\geq$ refers to inequality between all coefficients of corresponding powers of $z$.
\end{corollary}

\begin{proof}
This follows from Theorem \ref{thm:bijection} and the injection $\mathcal{R}^\p(L) \hookrightarrow T$ sending $\tau \mapsto (\tau, \sigma_0)$, where $\sigma_0$ is the unique element of $\widetilde{\mathcal{R}}^\p(K,\emptyset)$.
\end{proof}

\subsection{Normal rulings of satellites and the Thurston--Bennequin number}

Corollary~\ref{cor:estimate} suggests the following definition.

\begin{definition}
Let $K \subset \R^3$ be a Legendrian knot and $L \subset J^1(S^1)$ a Legendrian link, each equipped with a $\Z/\p$-valued Maslov potential. Then $K$ is \textit{$L$-compatible} if $R_{S(K,L)}^\p(z) \neq R_L^\p(z)$.
\end{definition}

In this subsection, we discuss a correlation between $L$-compatibility and maximal Thurston--Bennequin number; later we show that $L$-compatibility is related to the existence of representations of the Chekanov--Eliashberg DGA.

Recall that a Legendrian knot $K \subset \R^3$ has positive and negative stabilizations $S_+(K)$ and $S_-(K)$ obtained by inserting a pair of consecutive cusps (a zigzag) into a strand of the front projection of $K$.  The stabilization is positive (resp. negative) if the new cusps have the orientation of $K$ running downward (resp. upward) along the cusp.  The Legendrian isotopy type of $S_{\pm}(K)$ depends only on the Legendrian isotopy type of $K$, and any Legendrian $K' \subset \R^3$ that is isotopic to $S_{\pm}(K)$ for some $K \subset \R^3$ is said to be {\it stabilized}.

\begin{remark}
\begin{itemize}
\item[(i)] Stabilization may be viewed as a special case of the
  Legendrian satellite construction since $S_{\pm}(K) = S(K, L_{\pm})$
  where $L_{\pm} \subset J^1(S^1)$ are pictured in Figure \ref{fig:SatStab}.
\item[(ii)]  Note that $\tb(S_{\pm}(K)) = \tb(K) -1$.
\end{itemize}
\end{remark}

\begin{figure}
\[
\begin{array}{ccc}\includegraphics[scale=.5]{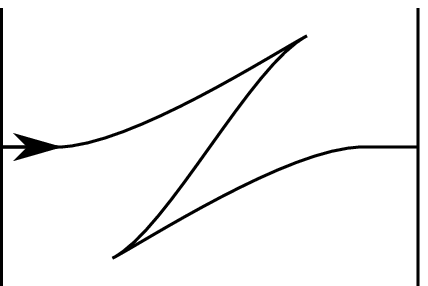} & \quad & \includegraphics[scale=.5]{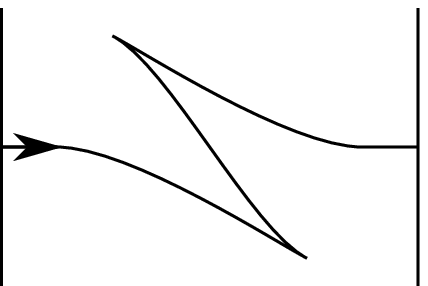} \\
L_+ & \quad & L_-
\end{array}
\]
\caption{Front projections of $L_+$ and $L_-$.}
\label{fig:SatStab}
\end{figure}

\begin{theorem} \label{prop:stab}
If $K$ is stabilized, then for any nonempty $L \subset J^1(S^1)$ with any choice of Maslov potential,
  $R_{S(K,L)}^\p(z) = R_L^\p(z)$, i.e., $K$ is not $L$-compatible.
\end{theorem}

\begin{proof}
We may assume that $(K, *)$ and $L$ are in general position; that the front diagram of $K$ contains a zigzag; and that the base point $*$ does not lie on the strand, $S$, that connects the two cusps of the zigzag.  On the front diagram of the satellite $S(K,L)$, let $C_1, \ldots, C_n$ and $D_1, \ldots, D_n$ denote the left and right cusps corresponding to the zigzag on $K$.

Let $\rho$ be a normal ruling of $S(K,L)$.  We will show that along $S$ the thin part of $\rho$, $\rho_T$, does not have fixed points.  Consider the ruling paths $P_i$ and $Q_i$ originating at a cusp $C_i$.  If $P_i$ switches before leaving the crossing near the $C_i$, then it follows from the Cusp Lemma that $\rho_T(i) \neq i$.  If not, then $P_i$ must end at one of the $D_j$.  Then  $Q_i$ would need to end at $D_j$ as well, but this ensures that $Q_i$ must switch at one of the crossings near the $C_i$.  Again, the Cusp Lemma shows that $\rho_T(i) \neq i$.

Now, from the Cusp and Crossing Lemmas we see that the thin part of $\rho$ must be fixed point free everywhere.  In particular,  $\rho$ restricts to a normal ruling, $\tau$, (without fixed points) on $L \subset J^1(S^1)$.  It follows that $\rho = \Phi(\tau, \sigma_0)$.  Thus, the injection of $\mathcal{R}^\p(L) \hookrightarrow T$ from Corollary \ref{cor:estimate} composed with $\Phi$ is onto $\mathcal{R}^\p(K,L)$, and the result follows.
\end{proof}

\begin{corollary} \label{cor:MaximizeTB} If there exists a pattern $L \subset J^1(S^1)$ such that $K$ is $L$-compatible, then $\tb(K)$  is maximal within the smooth knot type of $K$.
\end{corollary}

\begin{proof}  If $\tb(K)$ is non-maximal, then there exists a stabilized knot $K'$ with the same smooth type as $K$ and $\tb(K) =\tb(K')$.  Then
\[
R_{S(K,L)}^1(z) = R_{S(K',L)}^1(z) = R_{L}^1(z),
\]
where the first equality is a combination of  Remark \ref{rem:sat} with Theorem \ref{thm:Kauffman} and the second is Theorem \ref{prop:stab}.

If $K$ is $L$-compatible where $L$ has a $\p$-graded Maslov potential, then it remains $L$-compatible in the ungraded setting. This contradicts $R_{S(K,L)}^1 = R_L^1$.
\end{proof}

\subsection{Reduced ruling polynomials}

Theorem \ref{thm:bijection} shows that
\begin{equation} \label{eq:BijFormula}
R^\p_{S(K,L)}(z) = \sum_{\tau \in \mathcal{GR}^\p(L)} z^{j(\tau)} \widetilde{R}^\p_{S(K, L^\tau)}(z).
\end{equation}
We will use this relation to deduce properties of reduced ruling polynomials from corresponding results about standard ruling polynomials.

\begin{theorem} \label{prop:LInvariance} Let $K \subset \R^3$ be a Legendrian knot with $\p \, |\, 2 r(K)$.  For any fixed $\p$-graded $L \subset J^1(S^1)$ with generic front projection,   $\widetilde{R}^\p_{S(K, L)}$ is a Legendrian isotopy
  invariant of $K$.
\end{theorem}

\begin{proof}
First we establish the result for those front projections $L$ that do not have cusps.  In this case, the summation on the right side of (\ref{eq:BijFormula}) contains $\widetilde{R}^\p_{S(K,L)}$ (when $\tau$ is the generalized ruling that fixes every strand of $A_\Lambda$).
The remaining terms on the right hand side are a $\Z[z^{\pm 1}]$-linear combination of reduced ruling polynomials of the form $\widetilde{R}^\p_{S(K, L')}$ where $L'$ is again a front projection without cusps and fewer strands than $L$.  The front projections $L'$ that appear in the summation as well as the coefficients depend only on $L$.  Since the left hand side of (\ref{eq:BijFormula}) is a Legendrian isotopy invariant of $K$ (Theorem \ref{thm:CP}), the result follows  from inducting on the number of strands of $L$.

The general case when $L$ has cusps follows from this special case.  The reduced ruling polynomial of $L$ arises from restricting the sum on the right hand side of (\ref{eq:BijFormula}) to those generalized rulings $\tau$ that are fixed point free on the line $x=0$.  (This is due to the construction of the bijection in Theorem \ref{thm:bijection}.)  The remaining terms form a linear combination of reduced ruling polynomials of satellites of $K$ where the patterns do not have cusps.  (For any generalized ruling, the fixed point set $L^\tau$ cannot have cusps.)  Moreover, the particular patterns and coefficients appearing in this linear combination only depend on $L$.  These reduced ruling polynomials are all known to be Legendrian isotopy invariants of $K$, as is the left hand side of (\ref{eq:BijFormula}), so the result follows.
\end{proof}

\begin{remark} The reduced ruling polynomials
  $\widetilde{R}^\p_{S(K,L)}$ are not invariant under Legendrian
  isotopy of $L$.  For instance, an isotopy of $L$ that pushes a left cusp from the right side of $x=0$ to the left side of $x=1$ causes the reduced ruling polynomial to vanish.  Currently, we do not know if reduced ruling
  polynomials may be reformulated to obtain this property, and we
  leave the matter as an open question.
\end{remark}

\begin{theorem} \label{prop:RedTopInv} For $\p=1$ or $2$, $\widetilde{R}^\p_{S(K, L)}$ depends only on $L$, $\tb(K)$, and the underlying smooth knot type of $K$.
\end{theorem}

\begin{proof}
The proof follows a similar scheme to that used for Theorem \ref{prop:LInvariance}.  In this case, Theorem \ref{thm:Kauffman} together with Remark \ref{rem:sat} are used to establish the inductive step.
\end{proof}

\begin{example}
As an example, we consider the case of $1$-graded reduced rulings when the pattern is a product of basic
fronts, $A_\Lambda$, where $\Lambda =(\lambda, {\bf m})$.  Notation is as in Section 2.3.
In the case  of $1$-graded rulings, ${\bf m}$ is uninteresting, and it is not too hard to use (\ref{eq:BijFormula}) to give a rather explicit relation between the ruling polynomials of satellites $S(K, A_\lambda)$ and their reduced analogs.

\begin{theorem} \label{prop:AFormula} Given a partition $\lambda = (\lambda_1, \ldots,
  \lambda_\ell)$, let  $M_\lambda^{\mathit{sym}}$ denote the set of
  all symmetric $\ell \times \ell$ matrices with nonnegative integer
  entries with row sums and column sums equal to $\lambda$.  Then
\[\displaystyle
R^1_{S(K, A_\lambda)}(z) = z^{\ell (\ell-1) } \sum_{(b_{ij}) \in M_\lambda^{\mathit{sym}}} (\prod_i z^{-\delta_i}) (\prod_{i < j} \langle b_{ij} \rangle)  \widetilde{R}^1_{S(K, A_{(b_{11}, \ldots, b_{\ell\ell})})}(z)
\]
where $\delta_i$ is the Kronecker delta $\delta_{b_{ii}, 0}$, and $\langle m \rangle$ denotes the ruling polynomial $R^1_{A_m A_m}(z)$ if $m \neq 0$ and $z^{-2}$ if $m = 0$.
\end{theorem}

The proof is similar to that of Theorem 3.4 in \cite{LR} and is omitted here.  (Actually, we will not need such a precise formula.)  Inductively, Theorem \ref{prop:AFormula} can be used to give a formula for $\widetilde{R}^1_{S(K, A_\lambda)}(z)$ in terms of ordinary ruling polynomials.
\end{example}

The following result allows us to reduce questions of $L$-compatibility for general $L$ to the special case of the products $A_\Lambda$.

\begin{theorem} \label{thm:3equiv} Let $K$ be a Legendrian knot in $\R^3$ and $\p$ a divisor of $2 r(K)$.  Then the following are equivalent:

\begin{itemize}
\item[(i)] There exists a nonempty $\p$-graded product of basic fronts, $A_\Lambda$, with $\Lambda = (\lambda, {\bf m})$, such that $\widetilde{R}^\p_{S(K, A_\Lambda)} \neq 0$.

\item[(ii)] There exists a nonempty $\p$-graded product of basic fronts, $A_\Lambda$, with $\Lambda = (\lambda, {\bf m})$, such that $K$ is $A_\Lambda$-compatible.

\item[(iii)] There exists $L \subset J^1(S^1)$ with a $\p$-graded Maslov potential, such that $K$ is $L$-compatible.

\end{itemize}
Furthermore, any of these conditions imply that $K$ maximizes $\tb$.
\end{theorem}

\begin{proof}  The remark about $K$ maximizing $\tb$ is Corollary \ref{cor:MaximizeTB}.

The equivalence of (i) and (ii) follows from equation (\ref{eq:BijFormula}) with $L = A_\Lambda$.  Indeed, all of the summands on the right hand side have nonnegative coefficients, and the terms with $L^\tau = \emptyset$ produce exactly the polynomial $R^\p_{A_{\Lambda}}(z)$.  Since $L^\tau$ also has the form $A_{\Lambda'}$ for some $\Lambda'$, we see that these are the only non-zero terms if $\widetilde{R}^\p_{S(K, A_\Lambda)} = 0$ for all $\Lambda \neq 0$.  Hence, (ii) implies (i).  On the other hand,  $\widetilde{R}^\p_{S(K, A_\Lambda)}$ also appears in the sum when $\tau$ is the generalized ruling where every strand is a fixed point.  It follows that if $\widetilde{R}^\p_{S(K, A_\Lambda)} \neq 0$, then $R_{S(K,A_\Lambda)}^\p(z) > R_{A_\Lambda}^\p(z)$.  Thus, (i) implies (ii).

That (ii) implies (iii) is immediate.  For the converse, we need to recall some results about ruling polynomials.

The ruling polynomial $R^\p_L(z)$ satisfies skein relations as in Lemma 6.8 of \cite{R2} with the following modification from the $2$-graded case:  The coefficient $\delta_1$ (resp. $\delta_2$) is $1$ if the strands that cross in the first (resp. second) term have equal Maslov potential mod $\p$ and is $0$ otherwise.  (The proof is virtually identical.)  Moreover, the proof of Lemma 6.10 in the same article provides an algorithm for evaluating the ruling polynomial of an arbitrary Legendrian  $L \subset J^1(S^1)$ as a linear combination of ruling polynomials of basic fronts using these skein relations.  The algorithm, addressed to the $2$-graded case in \cite{R2}, applies equally well in the $\p$-graded case.  In addition, the ruling polynomial of a satellite, $R^\p_{S(K, L)}(z)$, satisfies the same skein relations in the factor $L$.  It follows that we can find coefficients $c_\Lambda(z) \in \Z[z^{\pm 1}]$ such that
\begin{equation}\label{eq:prf}
R^\p_L(z) = \sum c_\Lambda(z) R^\p_{A_\Lambda}(z) \quad  \mbox{and} \quad R^\p_{S(K;L)}(z) = \sum c_\Lambda(z) R^\p_{S(K;A_\Lambda)}(z).
\end{equation}

We now prove (the contrapositive of) (iii) implies (ii).  Assume that for all $\Lambda$, $R_{S(K;A_\Lambda)}(z) = R_{A_\Lambda}(z)$.   Combining this with (\ref{eq:prf}) shows that $R^\p_L(z) = R^\p_{S(K;L)}(z)$.
\end{proof}

\section{Finite-Dimensional Representations of $\A(K,*)$}
\label{sec:reps}

In this section we give necessary and sufficient conditions for the existence of finite dimensional representations of $\A(K, *)$ in terms of normal rulings of Legendrian satellites of $K$, including most of the main results mentioned in the Introduction.

\subsection{ The path matrix}
We begin with some linear algebra.
Let $L \subset J^1(S^1)$ be a $\p$-graded Legendrian link without cusps that intersects the line $x=0$ in $n$ points.  In this case, the DGA of the satellite $S(K,L)$ can be described in terms of the DGA of $(K,*)$ and a matrix $P_L$, known as the path matrix of $L$, which was introduced in \cite{Ka}.  The definition and properties of $P_L$ discussed here are all contained in \cite{Ka}.

Label the crossings of $L$ from left to right as $p_1,p_2, \ldots, p_r$.  As usual each crossing is assigned a degree in $\Z/\p$ as the difference of the value of the Maslov potential on the over- and understrands of the crossing, and the $p_i$ are viewed as non-commuting variables.

\begin{definition}  \label{def:PM} We consider paths within the front
  projection of $L$ that begin on the $i$-th strand (counting from top
  to bottom) at $x=0$ and end on the $j$-th strand (from top to
  bottom) at $x=1$.  At crossings we allow paths to either go straight
  through the crossing or turn a corner around the upper quadrant of
  the crossing.  To each such path we assign a word in the $p_i$ which
  is the product of crossings corresponding to corners of the path
  ordered from left to right (if there are no corners, the word is $1$).  The {\it path matrix} $P_L$ is the matrix whose $ij$-entry is the sum of words associated with all such paths.
\end{definition}

The path matrix is invertible as a matrix with entries in the non-commutative $\Z/2$-algebra generated by the $p_i$.  To see this note that
\[
P_L = C_1 C_2 \cdots C_r
\]
where $C_i$, $1\leq i\leq r$, is the invertible matrix equal to the
identity matrix except with the $2\times 2$ block
$\left[\begin{array}{cc} p_i & 1 \\ 1 & 0 \end{array}\right]$ placed
on the diagonal in rows $k$ and $k+1$, where $k,k+1$ are the labels of
 the strands involved in crossing $p_i$.
Using this perspective, it is also not hard to see that the entries of $(P_L)^{-1}$ correspond to paths from right to left in $L$ that are allowed to turn corners around the lower quadrant of a crossing.

\begin{example} \label{ex:basicpath}
For the basic front $A_n$, label the crossings in
  $A_n$ from left to right as $p_1,\ldots,p_{n-1}$; then the path
  matrix $P_{A_n}$ of $A_n$ satisfies:
\[
P_{A_n} = \left[ \begin{matrix} p_1 & p_2 & \cdots & p_{n-1} & 1 \\
1 & 0 & \cdots & 0 & 0 \\
0 & 1 & \cdots & 0 & 0 \\
\vdots & \vdots & \ddots & \vdots & \vdots \\
0 & 0 & \cdots & 1 & 0 \end{matrix} \right]
\text{~~~ and ~~~}
P_{A_n}^{-1} = \left[ \begin{matrix}
0 & 1 & 0 & \cdots & 0 \\
0 & 0 & 1 & \cdots & 0 \\
\vdots & \vdots & \vdots & \ddots & \vdots \\
0 & 0 & 0 & \cdots & 1 \\
1 & p_1 & p_2 & \cdots & p_{n-1}
\end{matrix} \right] .
\]
Note that $(P_{A_n}^{-1})^\mathrm{T}$ is a matrix in rational canonical form, the so-called
companion matrix to the polynomial $\lambda^n + p_{n-1} \lambda^{n-1}
+ \cdots + p_1 \lambda + p_0 \in (\Z/2)[\lambda]$.

More generally, if $\Lambda = (\lambda,{\bf m})$ and $\lambda =
(\lambda_1,\ldots,\lambda_\ell)$, then $(P_{A_\Lambda}^{-1})^\mathrm{T}$ is also a
matrix in rational canonical form: it is in block-diagonal form with
blocks $(P_{A_{\lambda_1}}^{-1})^\mathrm{T},\ldots,(P_{A_{\lambda_\ell}}^{-1})^\mathrm{T}$. For
future use, we record the following:

\begin{lemma} \label{lem:rcf} Any invertible matrix $M \in GL_n(\Z/2)$ is
conjugate to a matrix of the form $P_{A_\Lambda}$ for some $\Lambda$.
\end{lemma}
\begin{proof}
Any matrix at all, in particular $(M^{-1})^\mathrm{T}$, is conjugate to a block diagonal matrix with blocks of the form  $(P_{A_{\lambda_1}}^{-1})^\mathrm{T}$ except that in some blocks the $1$ in the upper right corner may be replaced by $0$.  (This is the standard rational canonical form found in most introductory algebra texts.)  Since $(M^{-1})^\mathrm{T}$ is invertible all of these entries must equal $1$, so $M$ has the desired form.
\end{proof}
\end{example}

\begin{example} \label{ex:HTW}
Let $\Delta_n$ denote a positive half twist of $n$ strands.  The path matrix of $\Delta_n$ is skew-upper triangular (i.e., all entries below the northeast-southwest ``antidiagonal'' are $0$), with $1$'s on the antidiagonal and variables $s_{ij}$ in the entries with $i+j \leq n$.  See Figure \ref{fig:HalfTw}.

\begin{figure}
\centerline{\raisebox{-6.2ex}{\includegraphics[scale=.8]{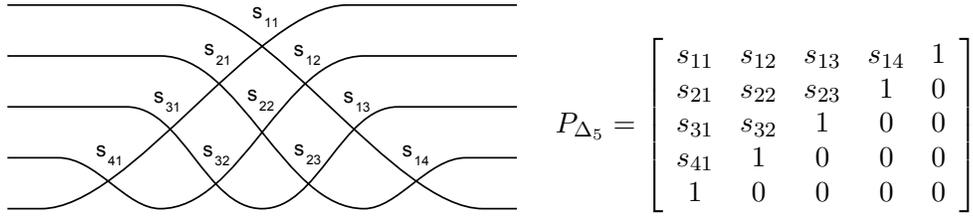}} \quad
$P_{\Delta_5} = \left[\begin{array}{ccccc} s_{11} & s_{12} & s_{13} & s_{14} & 1 \\
s_{21} & s_{22} & s_{23} & 1 & 0 \\
s_{31} & s_{32} & 1 & 0 & 0 \\
s_{41} & 1 & 0&0&0 \\
1&0&0&0&0 \end{array}\right]$}
\caption{A positive half twist with $n=5$ and its path matrix.
}
\label{fig:HalfTw}
\end{figure}

Given ${\bf n} : \Z/\p \rightarrow \Z_{\geq 0}$ we let $\mathit{tw}_{\bf n}$ denote a positive full twist of $n = \sum_{k\in \Z/\p} {\bf n}(k)$ strands  with $\p$-graded Maslov potential $\mu$ as follows:  The first ${\bf n}(0)$ strands have $\mu = 0$, the next ${\bf n}(1)$ strands have $\mu =1$, and continue in this manner until the last ${\bf n}(\p-1)$ strands have $\mu= \p-1$.
 The full twist is the concatenation of two half twists,  $\mathit{tw}_{\bf n} = \Delta_n * \Delta_n$, so it follows that the path matrix of $\mathit{tw}_{\bf n}$ is the product of two skew-upper triangular matrices, $S_1 S_2$.  Note that in $S_1$  (resp. $S_2$) the crossings with degree $0$ mod $\p$ are all on the blocks of sizes ${\bf n}(0) \times {\bf n}(0), \ldots, {\bf n}(\p-1) \times {\bf n}(\p-1)$ running along the antidiagonal and ordered from upper right to lower left (resp. lower left to upper right).
\end{example}

We conclude this discussion by proving two lemmas.  The first provides a standard form result related to the path matrices $P_{\mathit{tw}_{\bf n}}$, and the second involves normal rulings of satellites $S(K, \mathit{tw}_{\bf n})$.

\begin{lemma} \label{lem:factor}
Any matrix $M \in GL_n(\Z/2)$ is conjugate to a matrix of the form
\[
S_1 S_2 U,
\]
where $U$ is upper triangular and $S_1,S_2$ are skew-upper triangular.
\end{lemma}

\begin{proof}
We first prove the lemma when $M$ has the form $P_{A_n}^{-1}$ in the notation of
Example~\ref{ex:basicpath}. If we define $S_1,S_2,U$ by
\begin{gather*}
(S_1)_{ij} = \begin{cases} 1 & i+j=n \text{ or } n+1 \\
0 & \text{otherwise} \end{cases},
~~~~~~~~~~~(S_2)_{ij} = \begin{cases} 1 & i+j \leq n+1 \\
0 &\text{otherwise} \end{cases},\\
U_{ij} = \begin{cases} 1 & i=j \\
1+p_j & i=1 \text{ and } j>1 \\
0 & \text{otherwise} \end{cases},
\end{gather*}
then it is easy to check that $S_1S_2U = P_{A_n}^{-1}$.

For general $M$, up to conjugation, we may assume that $M^{\mathrm{T}}$ is in
rational canonical form, so that $M$ is conjugate to the block-diagonal form with blocks
$P_{A_{\lambda_1}}^{-1},\ldots,P_{A_{\lambda_\ell}}^{-1}$. Apply the
lemma to each of these blocks to obtain the desired $S_1,S_2,U$. (Note
that $S_1,S_2$ are also block-diagonal, with blocks corresponding to blocks of $M$ running along the
antidiagonal from top right to bottom left in $S_1$ and from bottom left to top right in $S_2$.)
\end{proof}

\begin{lemma} \label{lem:TwRul}  Every normal ruling of $S(K, \Delta_n)$ or $S(K, \mathit{tw}_{\bf n})$ is reduced.
\end{lemma}
\begin{proof}  We treat the case of a half twist; the  proof applies equally well to a full twist.

Following the labeling scheme for crossings indicated in Figure \ref{fig:HalfTw}, we modify the front projection of $\Delta_n$ by a planar isotopy so that crossings appear from left to right in the order $s_{n-1,1}, s_{n-2,1}, s_{n-2,2}, \ldots, s_{1,1}, \ldots, s_{1,n-1}$.  That is,
working from left to right the $(n-1)$-strand crosses over the $n$-strand, then the $(n-2)$-strand crosses over the two strands below it.  This continues, so that each strand takes its turn crossing over all of the strands below.

Suppose that we are given a non-reduced normal ruling of $S(K,\Delta_n)$.  From the Cusp and Crossing Lemmas, there will be at least one pair of thin strands entering the left side of $\Delta_n \subset S(K,\Delta_n)$.  To avoid intersecting each other, this pair of strands must be involved in at least one switch within $\Delta_n$ since every pair of strands crosses in $\Delta_n$.  Thus we can find a switch, $s$, within $\Delta_n$ with the property that at least one of the companion strands also lies within $\Delta_n$, and we can assume that the number of strands lying between the switching strand and its companion strand near the $x$-coordinate of the switch is minimized.

\begin{itemize}
\item[Case 1:]  This switching strand lies above its companion at $s$.  Then the strand must constitute the lower half of the switch.  (Suppose not; then, contrary to the minimality assumption, the normality condition would show that the companion strand to the lower half of the switch lies between the upper switching strand and its companion.)  According to the way we have arranged $\Delta_n$, heading to the right the switching strand will cross over its companion strand unless one of the two strands switches.  Again, the normality condition would force such a switch to contradict the minimality hypothesis.

\item[Case 2:]  This switching strand lies below its companion at $s$. The argument is symmetric.  This switching strand must be the upper half of the switch.  Heading left, it will intersect its companion strand unless a switch occurs.  Such a switch would contradict the minimality assumption.
\end{itemize}
\end{proof}

\subsection{Necessary and sufficient conditions for the existence of representations}
\label{ssec:repcond}

We are now in a position to state precisely our main results relating finite-dimensional representations and satellite rulings.

Given a $\p$-graded Legendrian link $L \subset J^1(S^1)$ we define a $(\Z/\p)$-graded $(\Z/2)$-vector space $V_L$. Suppose the front projection of $L$ intersects the vertical line $x=0$ at points with Maslov potential $\eta_1, \ldots, \eta_n \in \Z/\p$ from top to bottom.  Let $V_L$ be the vector space with basis $e_1, \ldots, e_n$ with basis vectors assigned the grading $|e_i| = \eta_i$ for $i = 1, \ldots n$.  The chosen basis $e_1, \ldots, e_n$ provides an isomorphism $\End(V_L) \cong \mathit{Mat}_{n\times n}(\Z/2)$, and in the following we make use of this identification to view elements of $\End(V_L)$ as matrices.  Note that $V_L$ has graded dimension ${\bf n}: \Z/\p \rightarrow \Z_{\geq 0}$ where ${\bf n}(k)$ is the number of strands of $L$ with Maslov potential equal to $k$ at $x=0$.

Our main technical statement relates reduced normal rulings of $S(K,L)$ to representations of $( \A(K,*), \partial)$ with underlying vector space, $V_L$.

\begin{theorem} \label{thm:tech}  Let $K \subset \R^3$ be a
  $\p$-graded Legendrian knot and $L \subset J^1(S^1)$ a $\p$-graded
  Legendrian link without cusps.  Then $S(K,L)$ has a $\p$-graded
  reduced normal ruling if and only if there exists a $\p$-graded
  representation
\[
f: (\A(K,*), \partial) \rightarrow (\End(V_L),0)
\] such that:
\begin{itemize}
\item
$f$ is a $(\Z/2)$-algebra map and $f\circ\partial = 0$;
\item
$f(t)$ is a matrix of the form $M_L U$, where $U$ is an upper
triangular $n\times n$ matrix, and $M_L$ is the image of
the path matrix
$P_L$ under an algebra map that sends each $p_i$ to some element of
$\Z/2$, with $p_i$ sent to $0$ unless $|p_i| \equiv 0 \pmod{\p}$.
\end{itemize}
\end{theorem}

The proof of Theorem \ref{thm:tech} is given in
Section~\ref{ssec:pftech}.  First we derive some consequences
which are some of the central results of this paper.

\begin{theorem} \label{thm:MainResult} Let $K$ be a Legendrian knot in $\R^3$ and $\p$ a divisor of $2 r(K)$.  Then for fixed, non-zero $\mathbf{n}: \Z/\p \rightarrow \Z_{\geq 0}$, the following are equivalent:
\begin{itemize}

\item[(i)] The Chekanov--Eliashberg algebra $\A(K, *)$ has a $\p$-graded representation of graded dimension $\mathbf{n}$.

\item[(ii)] There exists $\Lambda = (\lambda, {\bf m})$ with ${\bf n}_\Lambda = {\bf n }$  such that
$S(K,A_\Lambda)$ has a $\p$-graded reduced normal ruling.

\item[(iii)] The satellite of $K$ with a full positive twist, $S(K, tw_{\bf n})$, has a $\p$-graded normal ruling. (Note that this link is topologically is the $(\tb(K)+1)$-twisted $n$-copy of $K$.)
\end{itemize}
\end{theorem}

Note that when ${\bf n}$ is the map ${\bf n}(0) = 1$, ${\bf n}(k) = 0$
for $k\neq 0$, Theorem~\ref{thm:MainResult} reduces to
Theorem~\ref{thm:FIS}, the correspondence between existence of
$\p$-graded augmentations and $\p$-graded rulings.

\begin{proof}[Proof of Theorem~\ref{thm:MainResult}]
The forward direction of Theorem \ref{thm:tech} in conjunction with Lemma \ref{lem:TwRul} shows that either of (ii) or (iii) implies (i).

We now assume that $f: \A(K, *) \rightarrow \End(V)$ is a $\p$-graded representation with $\dim(V) = {\bf n}$, and prove that (ii) and (iii) hold.
From the definitions, $V = \oplus_{k \in \Z/\p} V_k$ with $\dim_{\Z/2}V_k= {\bf n}(k)$, and since $|t| = 0$, $f(t)$ has the form $\sum_{k \in \Z/\p} f(t)_k$ with $f(t)_k \in \mathit{GL}(V_k)$.

For (ii), we choose bases for each $V_k$, so that the matrix of $f(t)_k$ has the form described in Lemma \ref{lem:rcf}, and then concatenate these bases to produce a basis for $V$. Now, let $\lambda_1, \ldots, \lambda_\ell$ denote the block sizes and $m_1, \ldots, m_\ell \in \Z/\p$ denote the grading of the corresponding components.  Using notation as in the statement of Theorem \ref{thm:tech}, the choice of basis provides a grading preserving isomorphism $V \cong V_{A_{\Lambda}}$.  With this identification, we obtain a representation $f$
\[
f: \A(K, *) \rightarrow \End(V_{A_{\Lambda}}),
\]
and as discussed in Example \ref{ex:basicpath} $f(t)$ has the desired form $M_{A_{\Lambda}}$ so that we can apply Theorem \ref{thm:tech} to produce a reduced ruling of $S(K, A_\Lambda)$.

Now we establish (iii).  Applying Lemma \ref{lem:factor}, we can choose a basis for each $V_k$, $k = 0, \ldots, \p-1$, such that the matrix of $f(t)_k$ has the form $S^k_1 S^k_2 U^k$ with $S^k_i$ skew-upper triangular and $U^k$ upper triangular.  Concatenating these bases provides a grading preserving isomorphism $V \cong V_{\mathit{tw}_{\bf n}}$.   With respect to the distinguished basis, $f(t)$ has the form $S_1S_2U$ where $S_1$ (resp.  $S_2$) is obtained from placing the blocks $S^k_1$ (resp. $S^k_2$) along the antidiagonal from upper right to lower left (resp. lower left to upper right).  Consulting Example \ref{ex:HTW} we see that this matrix is of the form $M_{\mathit{tw}_{\bf n}}$ so that an application of Theorem \ref{thm:tech} completes the proof.
\end{proof}

In conjunction with Theorem \ref{prop:RedTopInv} and Corollary
\ref{cor:MaximizeTB}, Theorem~\ref{thm:MainResult} implies the following:

\begin{theorem}  \label{thm:topinvt}
The existence of a $1$- or $2$-graded representation of $\A(K, *)$ of any given dimension depends only on the smooth knot type of $K$ and $\tb(K)$.  If $\A(K, *)$ admits finite dimensional representations, then $K$ must have maximal Thurston--Bennequin number.
\end{theorem}

\begin{remark}  Shonkwiler and Vela-Vick \cite{SVV} gave examples of Legendrian knots in the topological knot types $m(10_{145})$ and $m(10_{161})$ with non-maximal Thurston--Bennequin and non-trivial Chekanov--Eliashberg algebras.  Sivek \cite{Siv} observed by a direct argument that these algebras do not admit finite-dimensional representations.  This may alternatively be deduced from Theorem \ref{thm:topinvt}.
\end{remark}

The following corollary of Theorem \ref{thm:3equiv} and Theorem \ref{thm:MainResult} addresses the problem of finding an arbitrary pattern $L$ that is compatible with $K$.

\begin{theorem} \label{thm:arbitrary}
The DGA $\A(K, *)$ has a $\p$-graded finite dimensional representation if and only if there exists a $\p$-graded pattern $L \subset J^1(S^1)$ so that $K$ is $L$-compatible.
\end{theorem}

We conclude this subsection by discussing the relation between
Theorem~\ref{thm:topinvt} and a conjecture from \cite{NgCLI} about
topological invariance of the abelianized characteristic algebra.

\begin{definition}[\cite{NgCLI}] Let $K$ be a Legendrian knot with DGA
  $(\A,\d)$. Define $I \subset \A$ to be the two-sided ideal generated
  by the collection $\{\d(a_i)\}$ of differentials of generators of
  $\A$. The \textit{characteristic algebra} of $K$ is the algebra
  $\A/I$; the \textit{abelianized characteristic algebra} of $K$ is
  the abelianization of $\A/I$.
\end{definition}

It was observed in \cite{NgCLI} that the abelianized characteristic
algebra, viewed without grading, often seems to depend only on the
smooth knot type of $K$ and $\tb(K)$, and it was conjectured that this
is always the case (up to a natural equivalence relation; see
\cite{NgCLI}). We do not
resolve this conjecture here, but
address a related construction.

\begin{definition}
Let $K$ be a Legendrian knot, and $(\A,\d),I$ as above. Define
$I'\subset\A$ to
be the smallest two-sided ideal containing $I$ such that whenever
$x,y\in\A$ satisfy $1-xy \in I'$, then $1-yx \in I'$ as well. The
\textit{partially abelianized characteristic algebra} of $K$ is
defined to be $\A/I'$.
\end{definition}

Note that we have a sequence of successive quotients: characteristic
algebra; partially abelianized characteristic algebra; abelianized
characteristic algebra.
From computations, it appears that the (ungraded) partially abelianized
characteristic algebra, like the abelianized version, depends only on
smooth type and $\tb$, and perhaps this is a more natural conjecture
than the conjecture from \cite{NgCLI}.

\begin{conjecture} \label{conj:paca}
Up to equivalence, the ungraded partially abelianized characteristic
algebra of a Legendrian knot $K$ depends only on the smooth type and
$\tb$ of $K$.
\end{conjecture}

Theorem~\ref{thm:topinvt} gives
some corroborating evidence:

\begin{corollary} \label{cor:paca}
Let $K_1,K_2$ be Legendrian knots with the same smooth type and
$\tb$. Then for any $n$, the ungraded partially abelianized
characteristic algebras of $K_1$ and $K_2$ either both have, or both
do not have, an $n$-dimensional representation over $\Z/2$.
\end{corollary}

It is conceivable
that this corollary could be strengthened to give some sort of correspondence between $n$-dimensional representations of $K_1$ and $K_2$, and that
the collection of finite-dimensional representations of a
$\Z/2$-algebra of the type considered here is enough to determine the
$\Z/2$-algebra, in which case Conjecture~\ref{conj:paca} would follow.
However, we do not pursue this direction further in this paper.

The rest of Section~\ref{sec:reps} is devoted to the proof of
Theorem~\ref{thm:tech}. We first address some preliminary issues
before presenting the proof in Section~\ref{ssec:pftech}.

\subsection{The DGA of the satellite $S(K, L)$}

For the rest of this section, we let $K \subset \R^3$ and $L \subset J^1(S^1)$ be $\p$-graded Legendrian links such that $L$ has no cusps.
We fix a base point $*$ on $K$ which is assumed to be located on a strand of $K$ that is oriented to the right.
For the proof of Theorem \ref{thm:tech}, there is no loss of generality from this assumption in view of Theorems \ref{prop:LInvariance} and \ref{prop:multi2}.  Furthermore, we may assume that the front diagram of $K$ is in plat position.  This can always be achieved by a Legendrian isotopy which will not effect the existence of either reduced rulings (Theorem \ref{prop:LInvariance}) or the type of representation in question.  (After an isotopy, $A(K, *)$ changes only by a stable tame isomorphism which maps $t$ to $t$.)

When $K$ is in plat position it follows that $S(K, L)$ will be plat as well.  We now describe the Chekanov--Eliashberg algebra associated with the resolution of this front diagram.  We maintain here our convention of, away from $L$, labeling the $n$ strands running parallel to $K$ from $1$ to $n$ as one moves from top to bottom.

\begin{figure}
\centerline{\includegraphics[scale=.5]{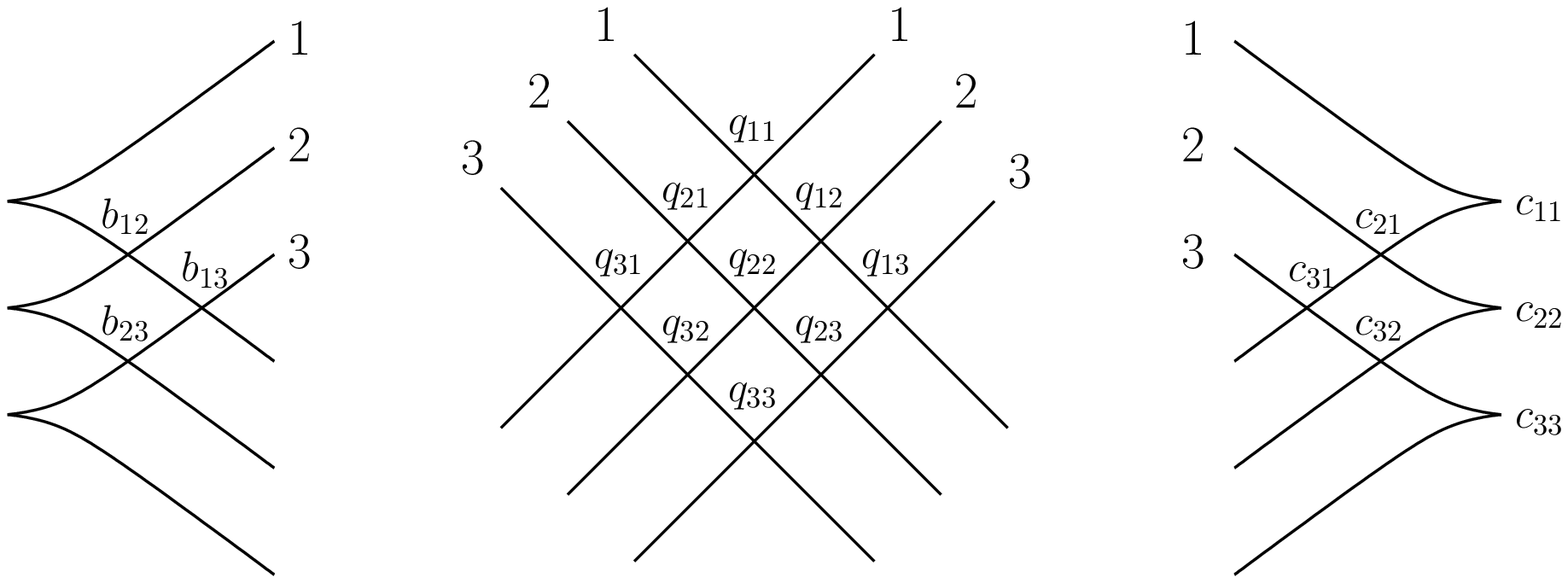}}
\caption{
Generators $b_{ij}^m$, $q_{ij}^m$, or $c_{ij}^m$ of $\A(S(K,L))$
corresponding to a left cusp, crossing, or right cusp $m$ of $K$; the
$m$ superscripts are suppressed.}
\label{fig:ncopy}
\end{figure}

\subsubsection{Generators of $\A(S(K,L))$}  We first label the
crossings and right cusps of $S(K,L)$; see Figure~\ref{fig:ncopy} for
an illustration.
Enumerate the left cusps of $K$ as $b_1, \ldots, b_{M_1}$.  Each of
these produces a strictly upper triangular matrix worth of generators
$b^m_{ij}$, $1 \leq i < j \leq n$, $m =1, \ldots, M_1$ where
$b^m_{ij}$ denotes the crossing of the $i$-th strand over the $j$-th
strand at the location of the cusp $b_m$.  Enumerate the crossings of
$K$ as $q_1, \ldots, q_{M_2}$.  For $1 \leq m \leq M_2$, there are
corresponding generators $q_{ij}^m$ with $1 \leq i, j \leq n$ where
the $i$-th strand crosses over the $j$-th strand at $q_m$.  Finally,
at each of the right cusps $c_1, \ldots, c_{M_3}$ there is a lower
triangular matrix worth of generators (note the absence of the term
{\it strictly}) $c^m_{ij}$ with $1 \leq j \leq i \leq n$ where the
$i$-th strand crosses over the $j$-th strand in the Lagrangian
projection.  Generators of the form $c^m_{ii}$ correspond to right
cusps of $S(K, L)$.

Finally, there are generators corresponding to the crossings of $L$ which, following our convention for the path matrix, we enumerate as $p_1, p_2, \ldots, p_r$ from left to right.

\subsubsection{Grading} \label{sub:grading} For $j =1 , \ldots, n$, let $\eta_j \in \Z/\p$ denote the value of the Maslov potential for $L$ on the $j$-th strand at $x=0$.   The generators of $\A(S(K;A_\Lambda))$ have  degrees related to the degrees of generators of $\A(K, \*)$ as follows:
\[
\begin{array}{ccc} |b^m_{ij}| = \eta_i - \eta_j -1 & \quad & |c^m_{ij}| =  \eta_i - \eta_j + |c_m|    \\
									 |q^m_{ij}| = \eta_i - \eta_j + |q_m|. & \quad & \quad
\end{array}
\]
The crossings $p_i$ arising from $L$ have their degrees determined by the Maslov potential of $L$ as discussed above Definition \ref{def:PM}.

\subsubsection{Differential}
After a short preparation we will describe the differential for $\A\left(S(K, L)\right)$.
We collect generators corresponding to the left cusps, crossings, and right cusps of $K$ into matrices $B_{m}$, $Q_{m}$, and $C_{m}$ respectively.  The matrices $C_m$ (resp. $B_m$) are  lower triangular (resp. strictly upper triangular).
Let
\[
\Phi_L: \A(K, *) \rightarrow \mathit{Mat}_{n\times n}\left( \A\left(S(K, L)\right)\right)
\]
denote the algebra homomorphism which takes generators $q_m$ and $c_m$ to the corresponding matrices $Q_m$ and $C_m$ and takes $t$ to the path matrix $P_L$.

In our notation we will use $\partial$ to denote the differential of $\A(K, *)$, and $D$ for the differential in $\A\left(S(K,L)\right)$.  Moreover, let $\bar{D}: \mathit{Mat}_{n\times n}\left( \A\left(S(K, L)\right)\right) \rightarrow \mathit{Mat}_{n\times n}\left( \A\left(S(K, L)\right)\right)$ denote the map resulting from applying $D$ entry by entry.

\begin{theorem} \label{thm:matrixdiff} The differential $D: \A\left(S(K,L)\right) \rightarrow \A\left(S(K,L)\right)$ satisfies the following matrix formulas:
\begin{equation} \label{eq:DBm}
\bar{D} B_m = (B_m)^2
\end{equation}
\begin{equation} \label{eq:Q}
\bar{D} Q_m = \Phi_L( \partial q_m ) + O(B)
\end{equation}
\begin{equation} \label{eq:BB}
\bar{D} C_m =  \pi_{low}\circ\Phi_L( \partial c_m ) + O(B)
\end{equation}
In formulas (\ref{eq:Q}) and (\ref{eq:BB}), $O(B)$ denotes a matrix whose entries belong to the two-sided ideal generated by the $b^m_{ij}$.  In (\ref{eq:BB}), $\pi_{low}$ denotes the projection which replaces all of the entries above the main diagonal by $0$.

In addition, for any crossing $p_i$ arising from $L$, $Dp_i$ belongs to the two-sided ideal generated by the $b^m_{ij}$.
\end{theorem}

\begin{proof}
The formula (\ref{eq:DBm}) is easily seen explicitly.

In establishing (\ref{eq:Q}) and (\ref{eq:BB}) we follow Mishachev \cite{Mish} and divide the disks involved in computing the differential for $\A\left(S(K,L)\right)$ into two disjoint sets: {\it thin disks} and {\it thick disks}.  By definition, thin disks are entirely contained in the neighborhood of the front diagram of $K$ where the satellite construction is carried out, and all remaining disks are considered to be thick.
 (At each right cusp $c^m_{ii}$ there is also a disk not visible on the front projection that contributes $1$ to $d c^m_{ii}$.  These disks arise from the twists added near right cusps when converting a front projection to a Lagrangian projection via the resolution procedure.  We will consider such disks to be thick and refer to them later in the proof as {\it invisible disks}.)

As in \cite{Mish}, we consider a ``stick together map'' $s: \R^2 \rightarrow \R^2$ that retracts the neighborhood of the front projection of $K$ containing $S(K,L)$ onto the front projection of $K$ itself. The image of a thick disk under $s$ coincides with a disk $f \in \Delta (a; b_1, \ldots, b_\ell)$ involved in the computation of the differential in $\A(K, *)$.  Moreover, we have the following:

\medskip

\noindent {\bf Claim:} If $f \in \Delta (a; b_1, \ldots, b_\ell)$ has $w(f) = t^{\alpha_0}b_1t^{\alpha_1} \cdots b_\ell t^{\alpha_\ell} \in \A(K,*)$, then the collection of thick disks corresponding to $f$ provides precisely the term $\Phi_L( t^{\alpha_0}b_1t^{\alpha_1} \cdots b_\ell t^{\alpha_\ell})$ (resp. $\pi_{\mathit{low}}\circ \Phi_L( t^{\alpha_0}b_1t^{\alpha_1} \cdots b_\ell t^{\alpha_\ell})$) in $\bar{D}\Phi_L(a)$ when $a$ is a crossing (resp. right cusp) of $K$.

\medskip

From this claim it follows that restricting the sum defining $D$ to thick disks produces precisely the first terms on the right hand side of (\ref{eq:Q}) and (\ref{eq:BB}).

To verify the claim, we need to consider ``lifts'' of $f: (D^2_\ell, \partial D^2_\ell) \rightarrow (\R^2, \pi_{xy}(K))$ to $\tilde{f} : (D^2_{\ell'}, \partial D^2_{\ell'}) \rightarrow (\R^2, \pi_{xy}(S(K,L)) )$ such that $s \circ \tilde{f} = f|_{D^2_{\ell'}}$.  Here, we use the notation from Section \ref{sec:CEDGA}: $r, s_1, \ldots, s_\ell$ are marked points along $\partial D^2$, and $D^2_\ell= D^2 \setminus \{r, s_1, \ldots, s_\ell\}$.  We need to allow the possibility that $\ell' \geq \ell$ since $\tilde{f}$ may have more negative corners than $f$ because of the additional possibility of negative corners at crossings of $L$.
For concreteness, let's assume that $a = q_m$ and $b_1=q_{m_1}, \ldots, b_\ell = q_{m_\ell}$ are all crossings.  Moreover, initially we treat the case in which $f(\partial D_\ell)$ does not contain the base point $*$ of $K$.

Such a lift $\tilde{f}$ arises from an appropriate choice of corners in $S(K,L)$.   If $\tilde{f}(r) = q^m_{ij}$, then since a neighborhood of $r$ must map to a $+$ corner at $q_m$ we see that the arc of $\partial D^2_\ell$ extending counter-clockwise from $r$ must initially map to the $i$-th copy of $K$ in $S(K,L)$.  (Recall that our subscripts $i,j$ always indicate the $i$-th copy of $K$ crossing over the $j$-th copy.)     When we arrive at the crossings of $S(K,L)$ corresponding to $q_{m_1}$ the boundary of $\tilde{f}$ remains on the $i$-th copy of $K$ and must turn around a $-$ corner of a crossing of the form $q^{m_1}_{i,k_1}$.  This  puts the next arc of $\partial D^2_\ell$ on the $k_1$-copy of $K$.  Similarly, the next corner of $\tilde{f}$ must be a $-$ corner at a crossing of the form $q^{m_2}_{k_1,k_2}$.  In total, $\tilde{f}$ has negative corners at $q^{m_1}_{i,k_1}, q^{m_2}_{k_1, k_2}, \ldots, q^{m_\ell}_{k_{\ell-1} j}$ where the choices of $1 \leq k_1, \ldots, k_{\ell-1} \leq n$ can be arbitrary.

We see that lifts of $f$ with initial positive corner at $q^m_{i,j}$ contribute the sum
\[
\sum_{k_1, \ldots, k_{\ell-1}} q^{m_1}_{i,k_1} q^{m_2}_{k_1, k_2}\cdots q^{m_\ell}_{k_{\ell-1} j}
\]
to $D q^m_{ij}$.  This is precisely the $ij$ entry of the product $Q_{m_1} \cdots Q_{m_\ell} = \Phi_L(q_{m_1} \cdots q_{m_\ell})$, so in this case the claim follows.

Now we consider how the collection of lifts $\tilde{f}$ changes when $f( \partial D^2_\ell)$ is allowed to intersect the base point.  Recall the notation $\gamma_0, \gamma_1, \ldots, \gamma_\ell$ for the images under $f$ of the circular arcs from $r$ to $s_1$, $s_1$ to $s_2$, $\ldots$, and $s_{\ell}$ to $r$ respectively.  Suppose $\gamma_i$ intersects $*$ positively.  Since we have assumed $*$ is located on a strand of $K$ which is oriented to the right, it follows that $\gamma_i$ crosses $*$ from left to right.  Moreover, since $\tilde{f}$ is orientation preserving, a neighborhood of $\gamma_i^{-1}(*)$ in $D^2_\ell$ maps to the region of $\R^2$ above $*$.  The corresponding portion of $\tilde{f}$, $\tilde{\gamma}_i$, will travel from left to right through the subset $L \subset S(K,L)$, with possibly some convex negative corners at crossings $p_i$ in $L$ at which the image of $\widetilde{f}$ necessarily covers the upper quadrant of $p_i$.  The $ij$-entry of the path matrix $P_L$ records precisely the products of negative corners that can result if $\tilde{\gamma}_i$ enters the subset $L$ along the $i$-th strand of $K$ and departs along the $j$-th strand of $K$.  This shows that, as desired, such an occurrence of $t$ in $w(f)$ requires placing $P_L = \Phi_L(t)$ between $Q_{m_{i-1}}$ and $Q_{m_i}$ when computing the terms in $\bar{D}Q_m$ which correspond to lifts of $f$.

A similar argument shows that appearances of $t^{-1}$ in $w(f)$ translate to $\Phi_L(t^{-1})$ in the computation of $\bar{D}Q_m$.  Indeed, if $\gamma_i$ intersects $*$ negatively, then $\tilde{\gamma}_i$ travels from right to left along $L$ with convex corners corresponding to the bottom quadrants of crossings of $L$.  As discussed after Definition \ref{def:PM} the matrix whose $ij$-entry corresponds to such paths is precisely $P_L^{-1} = \Phi_L(t^{-1})$.

The above analysis of thick disks applies equally well to establish the term $\pi_{\mathit{low}}\circ \Phi_L( t^{\alpha_0}b_1t^{\alpha_1} \cdots b_\ell t^{\alpha_\ell})$ in $\bar{D}C_m$.  The projection $\pi_{\mathit{low}}$ appears simply because there are no crossings $c^m_{ij}$ with $i <j$.  Note also that the invisible disks at right cusps $c^m_{ii}$ mentioned earlier in the proof contribute the identity matrix $I$ to $\bar{D}C_m$ and this corresponds to $\pi_{\mathit{low}}\circ \Phi_L$ applied to the $1$ in $\partial c_m$ arising from the invisible disk in $K$ at $c_m$.

 The only claim made in the Theorem about $Dp_i$ and the remaining terms of $\bar{D}Q_m$ and $\bar{D}C_m$ is that they belong to the two-sided ideal generated by the $b^m_{ij}$.  All of these items correspond to thin disks.  Therefore, the proof is completed by showing that any thin disk must have a negative corner at some $b^m_{ij}$.

In general, any disk other than an invisible disk at a right cusp will attain its minimum $x$-coordinate at a left cusp.  (As a consequence of the resolution construction negative corners can never provide a minimum $x$-coordinate and only the invisible disks have positive corners at the right quadrant of a crossing.)  With our assumption that $L$ has no cusps, the only left cusps of $S(K,L)$ are those corresponding to a left cusp $b_m$ of $K$.  They are accompanied on their right by a collection of crossings of the form $b^m_{ij}$.

Now, in order to reach its left cusp, the boundary of a thin disk will enter the collection of crossings $b^m_{ij}$ along two parallel strands of $S(K,L)$ that correspond to the same strand of $K$.  For these strands to meet up at a common left cusp of $S(K,L)$, one of them must have a negative corner at one of the $b^m_{ij}$.
\end{proof}

\begin{remark}  When $L$ consists of $n$ horizontal lines so that $S(K,L)$ is the $n$-copy we can be more explicit about the terms $O(B)$.  In this case, we have
\[
\bar{D} Q_m = B_{m'} Q_m + Q_m B_{m''} + \Phi_L( \partial q_m ), \mbox{ and}
\]
\[
\bar{D} C_m = B_{m'} C_m + C_m B_{m''} + \pi_{low}\circ\Phi_L( \partial c_m ).
\]
Here, $m'$ (resp. $m''$) are such that beginning at the upperstrand (resp. lowerstrand) of $q_m$ or $c_m$ and heading left the first left cusp reached will be $b_{m'}$ (resp. $b_{m''}$).
\end{remark}

\subsection{Proof of Theorem \ref{thm:tech}}
\label{ssec:pftech}

Throughout this subsection we fix notations as in the statement of Theorem \ref{thm:tech}:  the $\p$-graded vector space $V_L$ has basis vectors $e_1, \ldots, e_n$ where the degree of $e_i$ is given by the Maslov potential of $L$ at $x=0$ on the $i$-th strand.  Also, $M_L$ denotes a matrix obtained from the path matrix $P_L$ by assigning values in $\Z/2$ to the variables $p_i$ with the restriction that $p_i = 0$ unless $|p_i| = 0$ mod $\p$.

\begin{lemma}  The satellite $S(K,L)$ has a $\p$-graded reduced ruling if and only if there exists a $\p$-graded augmentation of $\A(S(K,L))$ such that $\varepsilon(B_m) = 0$ for all $m$.
\end{lemma}
\begin{proof}
This is an application of Theorem \ref{prop:bg}.
\end{proof}

Consider the Lagrangian projection of $K$ where in addition to our original base point, $* = *_0$, we add base points $*_1, \ldots, *_{M_3}$ (cyclically ordered) at all right cusps.  We use $\widehat{\partial}$ to denote the differential in the corresponding multi-pointed DGA for $K$.  (See Section \ref{sec:CEDGA}.)  The relation of $\widehat{\partial}$ with the differential $\partial$ of $\A(K, *)$ is simply
\[
\widehat{\partial}q_m = \partial q_m \quad \mbox{and} \quad \widehat{\partial}c_m = \partial c_m + 1 + t_m,
\]
provided that we identify $t_0$ and $t$ to view $\A(K,*)$ as the subalgebra $\A(K,*_0) \subset \A(K, *_0, \ldots, *_{M_3})$.
\begin{lemma}  There exists a $\p$-graded augmentation of $\A(S(K,L))$ such that $\varepsilon(B_m) = 0$ for all $m$ if and only if there exists has a $\p$-graded representation, $\widehat{f} :(\A(K, *_0, \ldots, *_{M_3}), \widehat{\partial}) \rightarrow (\End(V_L), 0)$, such that $t_0 \mapsto M_L$ and $t_i \mapsto U_i$ with $U_i$ upper triangular for all $i \geq 1$.
\end{lemma}

\begin{proof} $(\Rightarrow)$ Assume $\varepsilon : \A(S(K,L)) \rightarrow \Z/2$ is such an augmentation.  We let $\bar{\varepsilon}: \mathit{Mat}_{n\times n}\left( \A\left(S(K, L)\right)\right) \rightarrow \mathit{Mat}_{n \times n}(\Z/2) \cong \End(V_L)$ denote the algebra homomorphism resulting from applying $\varepsilon$ entry by entry.

We define the desired representation $\widehat{f} :\A(K, *_0, \ldots, *_{M_3}) \rightarrow \End(V_L)$ in two steps.  First, we require that on the subalgebra $\A(K, *_0) \subset \A(K, *_0, \ldots, *_{M_3})$ we have
\begin{equation} \label{eq:restrict}
\widehat{f}|_{\A(K, *_0)} = \bar{\varepsilon} \circ \Phi_L.
\end{equation}
Explicitly, $\widehat{f}(q_m) = \bar{\varepsilon}(Q_m)$; $\widehat{f}(c_m) = \bar{\varepsilon}(C_m)$; and $\widehat{f}(t_0) = \bar{\varepsilon}(P_L)$.  Note that since $\varepsilon$ is $\p$-graded $\widehat{f}(t_0)$ has the required form $M_L$.  On the remaining generators $t_m$ with $m=1, \ldots, M_3$ we define
\[
\widehat{f}(t_m) = I + \bar{\varepsilon} \circ \Phi_L( \partial c_m).
\]
The computation
\[
0 = \bar{\varepsilon}\circ \bar{D}(C_m) = \bar{\varepsilon}\circ \pi_{\mathit{low}} \circ \Phi_L (\partial c_m) = \pi_{\mathit{low}}(\bar{\varepsilon} \circ \Phi_L( \partial c_m))
\]
shows that $\widehat{f}(t_m)$ is non-singular and upper triangular as required.

To check that $\widehat{f}$ is $\p$-graded it suffices to show that $|\widehat{f}(s)| =|s|$ on any generator $s$.  (The grading on $\End(V_L)$ is defined as in Section \ref{ssec:rep}.)  First we treat the case when $s$ has the form $q_m$.  Note that the same argument applies when $s=c_m$.

  On a basis vector $e_j$ we have
\[
\widehat{f}(q_m) e_j = \sum_i \varepsilon(q^m_{i,j}) e_i.
\]
Since $\varepsilon$ is $\p$-graded, if $\varepsilon(q^m_{i,j}) \neq 0$ we have
\[
0 = |q^m_{i,j}| = \eta_i - \eta_j + |q_m|
\]
which shows that
\[
|e_j| - |e_i| = |q_m|.
\]
It follows from the definition that $\widehat{f}(q_m)$ has degree $|q_m|$ in $\End(V)$.

Next, we verify that $|\widehat{f}(t_0)| =0$.  Notice that if the $i, j$-entry of $\bar{\varepsilon}(P_L)$ is non-zero, then there is a path from $x=0$ to $x=1$ which starts on the $i$-th strand, ends on the $j$-th strand, and only turns along crossings which have equal Maslov potential.  Clearly, the Maslov potential is constant along such a path, so $\eta_i = \eta_j$ which implies that $\widehat{f}(t_0)$ preserves the grading on $V_L$.

At this point we have shown that $\widehat{f}$ is degree preserving on the sub-algebra $\A(K, *_0)$.  Since $\partial$ lowers degree by $1$ on $\A(K, *_0)$, it follows now that both the terms defining $\widehat{f}(t_m)$ have degree $0$.

It remains to show that $\widehat{f} \circ \widehat{\partial} = 0$, and it suffices to verify this equality on generators with the case of $t_i$ being immediate for $i = 0, \ldots, M_3$.  For a crossing $q_m$ of $K$, $\widehat{\partial}(q_m) = \partial(q_m) \in \A(K, *_0)$, so using (\ref{eq:restrict}), (\ref{eq:Q}), and the hypothesis that $\bar{\varepsilon}(B_m) = 0$ we can compute
\[
\widehat{f} \circ \widehat{\partial}(q_m) = \bar{\varepsilon} \circ\Phi_L \circ \partial(q_m) = \bar{\varepsilon} \circ \bar{D} (Q_m) = 0.
\]
For a cusp $c_m$ of $K$, we have $\widehat{\partial}(c_m) = \partial(c_m) + 1 + t_m$, so we compute
\[
\widehat{f} \circ \widehat{\partial}(c_m) = \bar{\varepsilon} \circ\Phi_L \circ \partial(c_m) + I + \widehat{f}(t_m) =  0.
\]
(The final equality is just the definition of $\widehat{f}(t_m)$.)

\smallskip

\noindent $(\Leftarrow)$  Suppose now that the representation $\widehat{f}$ is given.  We define $\varepsilon: \A(S(K,L)) \rightarrow \Z/2$ by requiring that the corresponding homomorphism of matrix algebras $\bar{\varepsilon}: \mathit{Mat}_{n\times n}\left( \A\left(S(K, L)\right)\right) \rightarrow \mathit{Mat}_{n \times n}(\Z/2)$ satisfies
\begin{equation} \label{eq:TheAug}
\begin{array}{ccc} \bar{\varepsilon}(B_m) = 0, & & \bar{\varepsilon}(Q_m) = \widehat{f}(q_m), \\
									\bar{\varepsilon}(C_m) = \pi_{\mathit{low}}( \widehat{f}(c_m)), & \mbox{and} & \bar{\varepsilon}(P_L) = M_L.
\end{array}
\end{equation}
These formulas uniquely specify $\varepsilon$ except possibly on the generators $p_i$ where the hypothesis on the matrix $M_L$ allows us to fix values for  $\varepsilon(p_i)$ which satisfy the matrix equation and have $\varepsilon(p_i) = 0$ unless $|p_i|=0$.
That $\varepsilon$ is $\p$-graded on generators of the form $q^m_{ij}$ and $c^m_{ij}$ is verified in a similar manner to the corresponding portion of the proof of the forward implication.

To complete the proof, we show that $\varepsilon \circ D =0$.  This is immediate on generators of the form $b^m_{ij}$ or $p_i$ since in either case $D$ applied to such a generator belongs to the two-sided ideal generated by the $b^m_{ij}$.  For the remaining generators it suffices to verify that $\bar{\varepsilon}\circ \bar{D} (Q_m) = \bar{\varepsilon}\circ \bar{D} (C_m) = 0$.

Let $\A'$ denote the subalgebra of $\A(K, *_0, \ldots, *_{M_3})$ generated by crossings $q_m$ and the original base point $*_0$.  Note that since $K$ is in plat position, we have $\partial q_m, \partial c_m \in \A'$ for any $m$.  Furthermore, Equation (\ref{eq:TheAug}) shows that $\widehat{f}|_{\A'} = \bar{\varepsilon} \circ \Phi_L|_{\A'}$.  Now, using these observations and (\ref{eq:Q}) we compute
\[
\bar{\varepsilon}\circ \bar{D} (Q_m) = \bar{\varepsilon}\circ \Phi_L( \partial q_m) = \widehat{f}(\partial q_m) = \widehat{f} \circ \widehat{\partial} q_m = 0.
\]
For a right cusp $c_m$, since $\partial c_m = \widehat{\partial} c_m + 1 + t_m$ we have
\[
\bar{\varepsilon}\circ \bar{D} (C_m) = \bar{\varepsilon}\circ \pi_{\mathit{low}} \circ \Phi_L( \partial c_m) =\pi_{\mathit{low}}\circ  \widehat{f}(\partial c_m) =
\]
\[
 \pi_{\mathit{low}}( \widehat{f} \circ \widehat{\partial} c_m + \widehat{f}(1 + t_m) ) = \pi_{\mathit{low}}(I + U_m) = 0.
\]
\end{proof}

\begin{lemma} The algebra $\A(K, *_0, \ldots, *_{M_3})$ has a $\p$-graded representation $\widehat{f}:(\A(K, *_0, \ldots, *_{M_3}), \partial) \rightarrow (\End(V_L), 0)$ such that $\widehat{f}(t_0)=M_L$ and $\widehat{f}(t_i)=U_i$ with $U_i$ upper triangular for all $i \geq 1$, if and only if there exists a  $\p$-graded representation $f: \A(K,*) \rightarrow (\End(V_L), 0)$ such that $f(t)=M_L U$ with $U$ upper triangular.
\end{lemma}

\begin{proof}
Given such a representation $\widehat{f}$ of $\A(K, *_0, \ldots, *_{M_3})$ we obtain the required representation of $\A(K, *)$ as the composition $f = \widehat{f} \circ \phi$ where $\phi : \A(K, *) \rightarrow \A(K, *_0, \ldots, *_{M_3})$ is the homomorphism guaranteed by Theorem \ref{prop:multi2}.

For the converse, assume that $f : \A(K, *) \rightarrow \End(V_L)$ is a representation with $f(t) = M_L U$.  Place base points $a_0, \ldots, a_{M_3}$ in a small neighborhood of the original base point $*$.  Then there is a DGA homomorphism $g : \A(K, *) \rightarrow \A(K, a_0, \ldots, a_{M_3})$ defined by fixing all generators other than $t$ and setting $g(t) = t_0 \cdots t_{M_3}$.  We define an algebra homomorphism $\widehat{f}:\A(K, a_0, \ldots, a_{M_3}) \rightarrow \End(V_L)$ by
\[
\widehat{f}(t_0) = M_L, \,\, \widehat{f}(t_1) = U, \,\, \widehat{f}(t_i) = I \mbox{ for $i>1$,}
\]
and  $\widehat{f}(s) = f(s)$ on the remaining generators.  Note that, $\widehat{f} \circ g = f$.  We verify that $\widehat{f}\circ \widehat{d}=0$ by checking the equality on generators with the case of the $t_i$ being immediate since $\widehat{d}(t_i) = 0$.  For any other generator, $s$, we can compute
\[
\widehat{f} \circ \widehat{\partial}(s) = \widehat{f} \circ \widehat{\partial}\circ g(s) = \widehat{f} \circ g \circ \partial(s) = f \circ  \partial (s) = 0.
\]

Finally, we obtain a representation of $\A(K, *_0, \ldots, *_{M_3})$ by composing $\widehat{f}$ with the isomorphism guaranteed by Theorem \ref{prop:multi1}.
\end{proof}

\section{Ungraded Two-Dimensional Representations of $\A(K,\ast)$}
\label{sec:2d}

As mentioned earlier, Theorem~\ref{thm:MainResult} generalizes
the known result that a Legendrian knot $K \subset \R^3$ has a
$\p$-graded ruling if and only if its DGA $(\A,\d)$ has a $\p$-graded
augmentation. In this section, we analyze the next simplest case of
Theorem~\ref{thm:MainResult}, which provides a correspondence between
rulings of doubles of $K$ and $2$-dimensional representations of
$(\A(K,\ast),\d)$. For the sake of simplicity, we will specialize to
the case $\p=1$, in which all representations and all rulings are
ungraded, and we will suppress any occurrences of $\p$ or of the
grading in our notation.

In this case, Theorem~\ref{thm:MainResult} states that
$(\A(K,\ast),\d)$ has an ungraded $2$-dimensional representation if
and only if $\widetilde{R}_{S(K,A_\Lambda)}(z) \neq 0$ for at least one of the two
partitions of $2$, $\Lambda
= (2)$ and $\Lambda = (1,1)$. For $\Lambda = (1,1)$, the only
generalized rulings of $A_{(1,1)}$ are the trivial ruling with no
fixed points (where the two strands of $A_{(1,1)}$ are paired) and the
trivial ruling with all fixed points. It follows from
Theorem~\ref{thm:bijection} or \eqref{eq:BijFormula} that we have
\[
R_{S(K,A_{(1,1)})}(z)  = 1 +
\widetilde{R}_{S(K,A_{(1,1)})}(z) = R_{A_{(1,1)}}(z) + \widetilde{R}_{S(K,A_{(1,1)})}(z).
\]
Thus $K$ is $A_{(1,1)}$-compatible if and only if
$S(K,A_{(1,1)})$ has a reduced ruling.

For $\Lambda=(2)$, the only generalized ruling of $A_{(2)}$ is the
trivial ruling where all points are fixed by the involution. From
\eqref{eq:BijFormula}, we have
\[
R_{S(K,A_{(2)})}(z) = \widetilde{R}_{S(K,A_{(2)})}(z)
= R_{A_{(2)}}(z) + \widetilde{R}_{S(K,A_{(2)})}(z).
\]
In this case as well, $K$ is $A_{(2)}$-compatible if and only if
$S(K,A_{(2)})$ has a reduced
ruling.

Theorem \ref{thm:tech} then yields the
following statement for $2$-dimensional representations.

\begin{theorem}
Let $K$ be a Legendrian knot in $\R^3$ with DGA $(\A(K,\ast),\d)$.
\label{thm:2dim}
\begin{enumerate}
\item
$K$ is $A_{(1,1)}$-compatible if and only if $(\A(K,\ast),\d)$ has a
two-dimensional (ungraded) representation sending $t$ to an upper
triangular $2\times 2$ matrix $\left(\begin{smallmatrix} 1 & * \\ 0 &
  1 \end{smallmatrix}\right) \in GL_2(\Z/2)$.
\item
$K$ is $A_{(2)}$-compatible if and only if $(\A(K,\ast),\d)$ has a
two-dimensional representation sending $t$ to a matrix $M_{A_2} U$,
where $U$ is upper triangular and $M_{A_2}$ is of the form
$\left(\begin{smallmatrix} * & 1 \\ 1 &
  0 \end{smallmatrix}\right)$.
\end{enumerate}
\end{theorem}

Given a representation of a DGA $(\A(K,\ast),\d)$ sending $t$ to a
matrix $M$, we can clearly construct a representation sending $t$ to any
matrix conjugate to $M$. In the group $GL_2(\Z/2)$, there are three
conjugacy classes, represented by
\[
I = \left( \begin{matrix} 1 & 0 \\ 0 & 1 \end{matrix} \right),
\hspace{5ex}
A = \left( \begin{matrix} 0 & 1 \\ 1 & 0 \end{matrix} \right),
\hspace{5ex}
B = \left( \begin{matrix} 0 & 1 \\ 1 & 1 \end{matrix} \right).
\]

Easy linear algebra yields the following corollary to
Theorem~\ref{thm:2dim}.

\begin{corollary}
Let $K$ be a Legendrian knot in $\R^3$ with DGA $(\A(K,\ast),\d)$.
\label{cor:2dim}
\begin{enumerate}
\item
$K$ is $A_{(1,1)}$-compatible if and only if $(\A(K,\ast),\d)$ has a
two-dimensional representation sending $t$ to either $I$ or $A$.
\item
$K$ is $A_{(2)}$-compatible if and only if $(\A(K,\ast),\d)$ has a
two-dimensional representation sending $t$ to either $A$ or $B$
(equivalently, sending $t$ to any invertible matrix besides $I$).
\end{enumerate}
Furthermore, both conditions depend only on the smooth type and
Thurston--Bennequin number of $K$, and
either of these conditions ensures that $K$ maximizes $tb$.
\end{corollary}

\begin{proof}
For the two numbered statements, enumerate the conjugacy classes in
$GL_2(\Z/2)$ represented by
matrices of the form $U$ or $M_{A_2}U$, with notation as in Theorem~\ref{thm:2dim}. For the final statement, refer to
Theorem~\ref{prop:RedTopInv} and
Theorem~\ref{thm:3equiv}.
\end{proof}

We now discuss $A_{(1,1)}$-compatibility and $A_{(2)}$-compatibility
separately. Note that any $K$ that has an ungraded ruling is
automatically $A_{(1,1)}$-compatible (just double the ruling in
$S(K,A_{(1,1)})$); it also has an ungraded augmentation and thus a
(reducible) two-dimensional representation sending $t$ to $I$.

\begin{figure}
\centerline{\includegraphics[scale=.5]{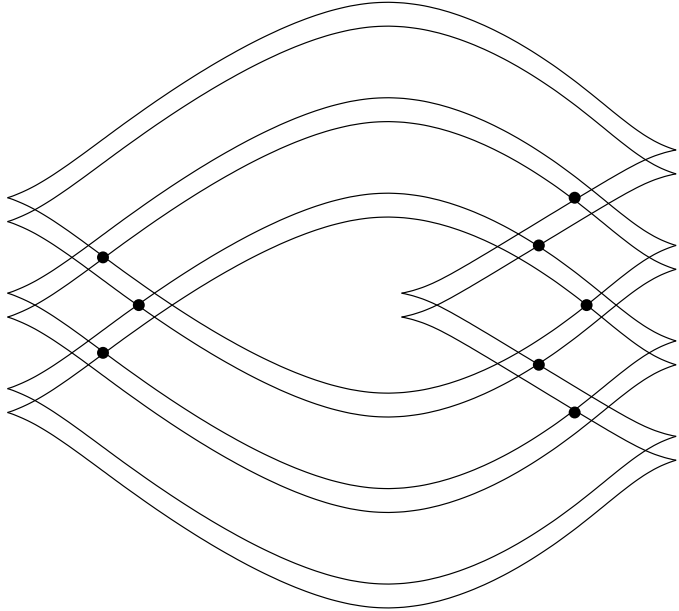}}
\caption{A ruling for the double of a negative torus knot. Dots
  represent switches. Shown here: $S(K,A_{(1,1)})$, where $K$ is a
  standard Legendrian $(3,-4)$ torus knot.}
\label{fig:TorusKnot}
\end{figure}

Sivek,
in his paper \cite{Siv} along with some other examples posted online, has constructed examples of Legendrian knots that have no
rulings but whose DGAs do have a (necessarily irreducible) two-dimensional
representation sending $t$ to $I$. These include Legendrian versions
of the torus knots
$T(p,-q)$ with $q>p\geq 3$ as well as $m(9_{42})$, $m(10_{128})$, and
$m(10_{136})$. By Corollary~\ref{cor:2dim}, each of these knots $K$ is
$A_{(1,1)}$-compatible and thus $S(K,A_{(1,1)})$ has a reduced
ruling. This last fact can be seen explicitly for the torus knots,
independent of Sivek's work; see Figure~\ref{fig:TorusKnot}, which
gives a reduced ruling when $K$ is a $(3,-4)$ torus knot, and which
can be readily generalized to any maximal-$tb$ Legendrian negative
torus knot.

Maximal-$tb$ Legendrian representatives of most knots with $10$ or
fewer crossings satisfy the Kauffman polynomial bound on $tb$ and thus
have ungraded rulings and augmentations. The exceptions are the knots
$T(3,-4) = m(8_{19})$, $m(9_{42})$, $m(10_{124}) = T(3,-5)$,
$m(10_{128})$, $m(10_{132})$, and $m(10_{136})$. Sivek's calculations,
along with Corollary~\ref{cor:2dim}, gives the following.

\begin{theorem}
Let $K$ be a Legendrian representative of a knot with
crossing number $\leq 10$. Then $K$ is $A_{(1,1)}$-compatible if and
only if both of the following hold: $K$ maximizes $tb$, and $K$ is not of
topological type $m(10_{132})$.
\end{theorem}

Note that the ``only if'' statement follows from Sivek's calculation
in \cite{Siv}
that the DGA for a particular maximal-$tb$ representative of
$m(10_{132})$ is trivial.

We remark that from our results, the
existence of a ruling for $S(K,A_{(1,1)})$ implies the
existence of a two-dimensional representation of $(\A(K,\ast),\d)$,
but not necessarily one that sends $t$ to $I$ as in Sivek's
examples; the representation might send $t$ to $A$. However, we do not
know of a knot where $S(K,A_{(1,1)})$ has a ruling but
$(\A(K,\ast),\d)$ has no two-dimensional representation sending $t$ to
$I$.

Finally, we turn to $A_{(2)}$-compatibility. Unlike for $A_{(1,1)}$,
the existence of an ungraded ruling of $K$ does \textit{not}
necessarily imply that $K$ is $A_{(2)}$-compatible. Indeed, the
standard Legendrian unknot with $tb=-1$ is not
$A_{(2)}$-compatible. However, a slightly stronger condition on
rulings of $K$ does imply $A_{(2)}$-compatibility.

\begin{theorem} \label{prop:A2}
If $K$ is a Legendrian knot in $\R^3$ with $\deg R_K(z) \geq 0$ (i.e.,
$K$ has a ruling where the number of switches is at least the number
of right cusps), then $K$ is $A_{(2)}$-compatible.
\end{theorem}

\begin{proof}
Consider a ruling of $K$ where the number of switches $s$ is at least the number
of right cusps $c$; this decomposes the front of $K$ into $c$
unknots. Construct a (planar) graph with $c$ vertices and $s$ edges, where the
vertices correspond to the unknots and edges correspond to
switches. Since $s \geq c$, this graph has a nonempty closed
loop. Thus we may choose some nonempty subset of ``distinguished''
switches such that every unknot contains an even number of
distinguished switches.

\begin{figure}
\centerline{\includegraphics[scale=.5]{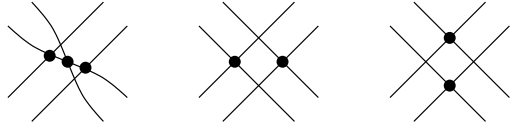}}
\caption{
Switches at doubled crossings in $S(K,A_{(2)})$: at the distinguished
switch $D$ (left); at every other distinguished switch (middle); and
at every non-distinguished switch (right).
}
\label{fig:switchA2}
\end{figure}

We use this information to construct a ruling of
$S(K,A_{(2)})$. Choose one distinguished switch $D$ of the
ruling of $K$, which we also view as a crossing in the front of
$K$. We construct the front for $S(K,A_{(2)})$ as follows: start with
the double of the front for $K$, so that every crossing in the front
for $K$ produces four crossings in the double; then place the extra
crossing for $A_{(2)}$ in the middle of the four crossings
corresponding to $D$, as shown in the leftmost diagram of
Figure~\ref{fig:switchA2}.

Now place switches at crossings of
$S(K,A_{(2)})$ as follows: do not switch at crossings corresponding to
cusps of $K$; at crossings corresponding to crossings of $K$, switch
according to Figure~\ref{fig:switchA2}. We leave it as an exercise to
the reader to check that this choice of switches on $S(K,A_{(2)})$
determines a ruling.
\end{proof}

We can use Theorem~\ref{prop:A2} to address the question of which
small Legendrian knots are $A_{(2)}$-compatible.
Recall from \cite{R} that the ungraded ruling polynomial $R_K(z)$
depends only on $tb(K)$ and
the Kauffman polynomial of the smooth knot underlying $K$.
An inspection of the Kauffman polynomial for smooth knots with up to
$10$ crossings shows that if $K$ is a maximal-$tb$ representative of a
smooth knot with at most $10$ crossings, then $\deg R_K(z) \geq 0$
unless $K$ is of one of the following types:
$0_1$, $m(8_{19})$, $m(9_{42})$, $m(9_{46})$, $m(10_{124})$,
$m(10_{128})$, $m(10_{132})$, $m(10_{136})$, and $m(10_{140})$. Of
these exceptions, a direct computation using \textit{Mathematica}
shows that maximal-$tb$ representatives of $m(8_{19})$, $m(9_{42})$,
$m(9_{46})$, $m(10_{124})$, $m(10_{128})$, and
$m(10_{136})$ are $A_{(2)}$-compatible: their
satellite with $A_{(2)}$ has an ungraded ruling. On the other hand,
maximal-$tb$ representatives of $0_1$, $m(10_{132})$, and
$m(10_{140})$ are not $A_{(2)}$-compatible. We summarize these
findings in the following result.

\begin{theorem}
Let $K$ be a Legendrian representative of a knot with crossing number
$\leq 10$.
Then $K$ is $A_{(2)}$-compatible if and
only if both of the following hold: $K$ has maximal $tb$, and $K$ is
not of topological type $0_1$, $m(10_{132})$, or $m(10_{140})$.
\end{theorem}


\begin{thebibliography}{}


\bibitem{Ch} Yu. V. Chekanov, Differential algebra of Legendrian
  links, \textit{Invent. Math.} \textbf{150} (2002), no. 3, 441--483.

\bibitem{ChP} Yu. Chekanov and P. Pushkar', Combinatorics of
Legendrian links and the Arnol'd 4-conjectures, \textsl{Uspekhi Mat. Nauk}
{\bf 60} (2005), no. 1, 99--154, translated in \textsl{Russian Math. Surveys}
{\bf 60} (2005), no. 1, 95--149.

\bibitem{EENS} T. Ekholm, J. Etnyre, L. Ng, and M. Sullivan, Knot
  contact homology, arXiv:1109.1542.

\bibitem{ENS} J. B. Etnyre, L. L. Ng, and J. M. Sabloff, Invariants of
  Legendrian knots and coherent orientations, \textit{J. Symplectic
    Geom.} \textbf{1} (2002), no. 2, 321--367.

\bibitem{F} D. Fuchs, The Chekanov--Eliashberg invariant of Legendrian knots: Existence of augmentations, \textsl{J. Geom. Phys.} {\bf 47 } (2003), no. 1, 43--65.

\bibitem{FI}
D. Fuchs and T. Ishkhanov, Invariants of Legendrian knots and
decompositions of front diagrams, \textit{Mosc. Math. J.} \textbf{4}
(2004), no. 3, 707--717, 783.


\bibitem{Henry2011}  M. B. Henry, Connections between Floer-type invariants and Morse-type invariants of Legendrian knots, \textit{Pacific J. Math.} {\bf 249} (2011), no. 1, 77-133.

\bibitem{Ka} T. K\'alm\'an, Braid-positive Legendrian links,
  \textit{Int. Math. Res. Not.} \textbf{2006}, Art. ID 14874.

\bibitem{LR} M. Lavrov and D. Rutherford, Generalized normal rulings and invariants of Legendrian solid torus links, \textit{Pacific J. Math.}, to appear; arXiv:1109.1319.

\bibitem{Mish} K. Mishachev, The $N$-copy of a topologically trivial
  Legendrian knot, \textit{J. Symplectic Geom.} \textbf{1} (2002),
  no. 4, 659--682.

\bibitem{NgCLI} L. L. Ng, Computable Legendrian invariants,
  \textit{Topology} \textbf{42} (2003), no. 1, 55--82.

\bibitem{NgLSFT}  L. Ng, Rational Symplectic Field Theory for
  Legendrian knots, \textit{Invent. Math.} \textbf{182} (2010), no. 3,
  451--512.


\bibitem{NgSabloff} L. Ng and J. Sabloff, The correspondence between augmentations and rulings for Legendrian knots, \textit{Pacific J. Math.} {\bf 224} (2006), no. 1, 141-150.

\bibitem{NgTr} L. Ng and L. Traynor, Legendrian solid-torus links, \textsl{J. Symplectic Geom.} {\bf 2} (2005), no. 3, 411-443.


\bibitem{R}  D. Rutherford, The Thurston--Bennequin number, Kauffman polynomial, and ruling invariants of a Legendrian link: The Fuchs conjecture and beyond, \textsl{Int. Math. Res. Not.} \textbf{2006}, Art. ID 78591.

\bibitem{R2} D. Rutherford, HOMFLY-PT polynomial and normal rulings of
  Legendrian solid torus links, \textit{Quantum Topol.} \textbf{2}
  (2011), 183--215.


\bibitem{Sabloff} J. M. Sabloff, Augmentations and rulings of
  Legendrian knots, \textit{Int. Math. Res. Not.} \textbf{2005},
  no. 19, 1157--1180.

\bibitem{SVV} C. Shonkwiler and D. S. Vela-Vick, Legendrian contact homology and nondestabilizability, \textit{J. Symplectic Geom.} {\bf 9} (2011), no. 1, 1-12.

\bibitem{Siv} S. Sivek, The contact homology of Legendrian knots with
  maximal Thurston--Bennequin invariant, \textit{J. Symplectic Geom.},
  to appear; arXiv:1012.5038.







\end{thebibliography}
\end{document}